\documentclass{amsart}

\input{defs.sty}

\bibliography{representability}

\title{A remark on Lurie's Representability Theorem}
\author{Aron Heleodoro}
\date{\today}

\begin{document}
    
\begin{abstract}
    In this note we revisit Lurie's representability theorem for geometric stacks and prove that one of the conditions can be mildly relaxed. The proof uses ideas from Hall--Rydh's work on the (classical) Artin's representability theorem. We also spell out a couple of variants of the main theorem and review the relation between deformation theory and homogeneity conditions on prestacks.
\end{abstract}

\maketitle

\tableofcontents

\section{Introduction}

The objective of this note is to prove the following result:

\begin{thm}[Main Theorem]
\label{thm:main-result}
Consider a map $p:\sX \ra S$, where $\sX$ is a prestack and $S$ is a Grothendieck affine scheme, then
\[
    \sX\mbox{ is an }n\mbox{-geometric stack locally almost of finite presentation}
\]
if and only if the following conditions are satisfied:
\begin{enumerate}
    \item $\sX$ is a sheaf with respect to the \'etale topology;
    \item $\sX$ is locally almost of finite presentation;
    \item $\sX$ is integrable;
    \item[$(iv)_{\rm triv}$] $\sX$ is $W^{\rm lA}_{\rm triv.}$-homogeneous (see Definition \ref{defn:P-homogeneous} below);
    \item[$(v)_{\rm ft}$] for every $(T \overset{x}{\ra} \sX) \in (\Schaffconvft)_{/\sX}$, $\sX$ admits a cotangent complex at $x$ (see \S \ref{subsubsec:convergent-and-laft-cotangent-complex} below);
    \item $\sX$ is convergent;
    \item[$(vii)_n$] the underlying classical prestack $\classical{\sX}$ is $n$-truncated.
\end{enumerate}
\end{thm}

\begin{rem}
Theorem \ref{thm:main-result} is very closely related to the celebrated Lurie's representability theorem (see \cite[Theorem 7.1.6]{DAG} or Theorem \ref{thm:representability-natural} below). Essentially the only difference is in condition (iv)\textsubscript{triv}, whereas Lurie's result imposes the stronger
\begin{enumerate}
    \item[$(iv)'_{\rm all}$] $\sX$ is infinitesimally cohesive (see Definition \ref{defn:infinitesimally-cohesive}).
\end{enumerate}
\end{rem}

\begin{rem}
In view of Proposition \ref{prop:equivalent-formally-cohesive} we could substitute condition (iv)\textsubscript{triv} in Theorem \ref{thm:main-result} by 
\begin{enumerate}
    \item[$(iv)'_{\rm Art}$] $\sX$ is formally cohesive (see Definition \ref{defn:formally-cohesive})
\end{enumerate}
\end{rem}

\begin{rem}
Notice that by Lemma \ref{lem:cotangent-complex-from-ft-points} assuming conditions (ii) and (vi), condition (v)\textsubscript{ft} is equivalent to the stronger
\begin{enumerate}
    \item[$(v)_{\rm all}$] $\sX$ admits a cotangent complex (see \S \ref{subsubsec:cotangent-complex});
\end{enumerate}
\end{rem}

\begin{parag}
Our proof uses exactly the same strategy as Lurie's result with the added observation learned from \cite{Hall-Rydh-algebraicity} that one can obtain condition (iv)'\textsubscript{all} from the weaker (iv)\textsubscript{triv} if we assume (i-iii), (v)\textsubscript{field} and that the diagonal is representable by a $0$-geometric stack (see Lemma \ref{lem:main-lemma} for the main technical input in weakening condition (iv)'\textsubscript{all}). However, notice that by setting up the inductive argument correctly we don't need to assume that the diagonal is geometric in Theorem \ref{thm:main-result}.
\end{parag}

\begin{rem}
Throughout this work we will consider the base as having characteristic $0$, this assumption bypasses many interesting complications. The main reason why we stick with it at the moment is that for arbitrary connective $\bE_{\infty}$-rings there are two nonequivalent notions of smoothness: fiber smooth and differentially smooth. Smoothness enters the main result in two points: 1) the very definition of $n$-geometric, which assumes one has smooth atlases, and 2) in the bootstrapping of deformation theory (Lemma \ref{lem:main-lemma}) which uses the fact that a smooth morphism between geometric stacks lifts against Henselian pairs (Lemma \ref{lem:Henselian-lift-against-smooth-geometric-stacks}).
\end{rem}

\subsection{Conventions}

Since the statement of Theorem \ref{thm:main-result} really depends on which conventions one is taking, we will be a bit pedantic and define all the terms involved.

\begin{parag}
A category, functor, or any categorical operation should be understood in the $\infty$-categorical sense. Given two objects $C,D$ in an $\infty$-category $\sC$ we let $\Map_{\sC}(C,D)$ denote their mapping space and $\Hom_{\sC}(X,Y) := \pi_0(\Map_{\sC}(C,D))$.
\end{parag}

\begin{parag}
A ring will always mean a commutative unital connective $\bE_{\infty}$-ring. When we restrict the base $S = \Spec\, k$ to the spectrum of a field of characteristic $0$, a ring can also be understood as a simplicial commutative $k$-algebra or a commutative differential graded $k$-algebra.
\end{parag}

\begin{parag}
A prestack is a functor from the opposite of the category of affine schemes to the category of spaces, i.e.\ the $\infty$-category freely generated under colimits by a single object.
\end{parag}

\begin{parag}
We adopt cohomological grading conventions everywhere.
\end{parag}

\begin{parag}
An affine scheme $S = \Spec\,R$ is \emph{Grothendieck} if $R$ is Grothendieck, i.e. 
\begin{defnlist}
    \item $R$ is Noetherian;
    \item $H^0(R)$ is a Grothendieck ring\footnote{Also refereed to as a  G-ring.}, i.e.\ for every prime ideal $\fp_0 \subset H^0(R)$ the canonical morphism $H^0(R)_{\fp_0} \ra H^0(R)^{\wedge}_{\fp_0}$ to its completion is regular.
\end{defnlist}
\end{parag}

\begin{parag}
\label{subsubsec:0-geometric-defn}
A prestack $p:\sX \ra S$ is \emph{$0$-geometric} if
\begin{defnlist}
    \item (descent) $\sX$ is an \'etale sheaf;
    \item (truncatedness) $\classical{\sX}$ is $0$-truncated (see \S \ref{subsubsec:truncated-prestack} below);
    \item (\'etale atlas) there exist $\sU \simeq \sqcup_I U_i$ a disjoint union of affine schemes and an affine representable morphism
    \[
        f:\sU \ra \sX
    \]
    such that $f$ is smooth\footnote{For any affine scheme $T \ra \sX$ the base change $\sU\underset{\sX}{\times} T \ra T$ is a smooth morphism of affine schemes.}.
\end{defnlist}
\end{parag}

\begin{parag}
A morphism $f:\sX \ra \sY$ is \emph{$0$-geometric}, if for every affine scheme $S \ra \sY$ the pullback $\sX\underset{\sY}{\times}S \ra S$ is $0$-geometric. A $0$-geometric morphism of prestacks $f$ is \emph{smooth} (resp.\ \emph{\'etale}) if for every affine point $S \ra \sY$ the pullback $\sX\underset{\sY}{\times}S \ra S$ is smooth (resp.\ \'etale), i.e.\ for any atlas $\sU=\sqcup_I U_i \ra \sX\underset{\sY}{\times}S$ each of the composites $U_i \ra S$ is smooth (resp.\ \'etale).
\end{parag}

\begin{parag}
\label{subsubsec:n-geometric-defn}
For $n \geq 1$ a prestack $p:\sX \ra S$ is \emph{$n$-geometric} if
\begin{defnlist}
    \item $\sX$ is an \'etale sheaf;
    \item the diagonal morphism $\Delta_{\sX,S}:\sX \ra \sX\underset{S}{\times}\sX$ is representable by an $(n-1)$-geometric stack\footnote{I.e.\ the base change with respect to any affine point is $(n-1)$-geometric.};
    \item (smooth atlas) there exist $\sU \simeq \sqcup_I U_i$ a disjoint union of affine schemes and a morphism
    \[
        f:\sU \ra \sX
    \]
    such that $f$ is $(n-1)$-representable and smooth.
\end{defnlist}
\end{parag}

\begin{parag}
A prestack (or morphism of prestacks) is said to be \emph{geometric} if it is $n$-geometric for some $n \geq 0$.
\end{parag}

\begin{parag}
A morphism $f:\sX \ra \sY$ of prestacks is \emph{$n$-geometric} if for every affine scheme $S \ra \sY$ the pullback $\sX\underset{\sY}{\times}S \ra S$ is $n$-geometric, and an $n$-geometric morphism $f$ of prestacks is \emph{smooth} (resp.\ \emph{\'etale}) if for every affine point $S \ra \sY$ the pullback $\sX\underset{\sY}{\times}S \ra S$ is smooth (resp.\ \'etale), i.e.\ for any atlas $\sU=\sqcup_I U_i \ra \sX\underset{\sY}{\times}S$ each of the composites $U_i \ra S$, which is $(n-1)$-geometric, is smooth (resp.\ \'etale).
\end{parag}

\begin{rem}
These definitions satisfy the usual compatibility, e.g.\ an $n$-geometric prestack (or morphism) is also $(n+1)$-geometric and similarly for the notion of smoothness\footnote{Cf.\ \cite[Chapter 2, \S 4.2]{GRI} for a similar discussion.}.
\end{rem}

\begin{rem}
Our convention regarding geometricity follows the definition taken in \cite[Chapter 5]{DAG}. In particular, its relation to the concept of $n$-Artin stack of \cite{GRI} is as follows: 
\[
    \mbox{For $n\geq 0$, any $n$-Artin stack is $n$-geometric and any $n$-geometric stack is $(n+1)$-Artin.}
\]
\end{rem}

\begin{rem}
Also notice that regarding Lurie's definition in \cite[\S 1.4 and \S 1.6.8]{SAG} we have that a $0$-geometric stack is equivalent to a spectral Deligne--Mumford $0$-stack, i.e.\ a spectral Deligne--Mumford $\sX$, whose underlying classical prestack $\classical{\sX}$ is $0$-truncated (see\cite[Theorem 1.7]{Porta-Comparison}).
\end{rem}

\begin{parag}
A prestack $\sX$ is said to be a sheaf for the \'etale topology if for every \'etale cover $T \ra S$ the canonical morphism
\[
    \sX(S) \ra \lim_{\Delta^{\rm op}}\sX((T/S)^{\bullet})
\]
is an equivalence, where $(T/S)^{\bullet}$ denotes the \v{C}ech nerve of $T \ra S$.

\end{parag}

\begin{parag}
\label{subsubsec:fp-convention-n-coconnective}
Let $\Schaffnfp$ denote the category of \emph{$n$-coconnective affine schemes of finite presentation}, i.e.\ the spectrum of rings $A$ which are compact objects in the category of $n$-truncated $\tau^{\geq -n}(R)$-algebras.
\end{parag}

\begin{parag}
\label{subsubsec:fp-convention-general}
A morphism of affine schemes $\Spec\,A = T \ra S =\Spec\,R$ is of \emph{finite presentation} if $A$ is a retract of a finite colimit of $R^N$ for some $N \geq 0$, equivalently $T$ is a co-compact object of $\Schaff_{/S}$. We let $\Schafffp$ denote the category of affine scheme of finite presentation over $S$.
\end{parag}

\begin{parag}
\label{subsubsec:afp-convention}
An affine scheme $\Spec\,A = T \ra S =\Spec\,R$ is \emph{almost finite presentation} if $\tau^{\leq n}(T) := \Spec\,\tau^{\geq -n}(A)$ is of finite presentation for every $n\geq 0$. We let $\Schaffafp$ denote the category of affine schemes almost of finite presentation over $S$. 
\end{parag}

\begin{parag}
\label{subsubsec:fp-for-Noetherian-base}
When the base $S = \Spec\,R$ is Noetherian, i.e.\ 
\begin{condlist}
    \item $H^0(R)$ is Noetherian,
    \item each $H^i(R)$ is a finitely generated $H^0(R)$-module;
\end{condlist}
the above finiteness conditions simplify as follows:

\begin{itemize}
    \item We refer to condition \ref{subsubsec:fp-convention-n-coconnective} on $n$-coconnective affine schemes as \emph{finite type} and denote its category by $\Schaffnft$. Concretely, $T=\Spec\,A \in \Schaffnft$ if:
    \begin{condlist}
        \item $H^0(A)$ is a finitely generated discrete $H^0(R)$-algebra;
        \item $H^i(A)$ is a finitely generated $H^0(A)$-module;
        \item $H^{-i}(A) = 0$, for every $i > n$.
    \end{condlist}
    \item We refer to condition \ref{subsubsec:fp-convention-general} as \emph{affine schemes of finite type} and denote the category of such by $\Schaffft$. \item We refer to condition \ref{subsubsec:afp-convention} as \emph{affine schemes almost of finite type} and denote this category by $\Schaffaft$.
\end{itemize}
\end{parag}

\begin{parag}
\label{subsubsec:eventually-coconnective-ft-affine-schemes}
In the case of a Noetherian base, we will often use the subcategory $\Schaffconvft \subset \Schaffft$ of \emph{eventually coconnective affine schemes of finite type}, i.e.\ $T \in \Schaffconvft$ if $\tau^{\geq n}(T) \overset{\simeq}{\ra}T$ for some $n \geq 0$. Notice that when $S$ is itself eventually coconnective, e.g.\ the spectrum of a field, this condition is automatic.
\end{parag}

\begin{parag}
\label{subsubsec:lafp-prestack}
A prestack $\sX$ is \emph{locally almost of finite presentation} if for every $n \geq 0$ the restriction
\begin{equation}
    \label{eq:n-truncation-prestack-lfp}    
    {^{\leq n}\sX}:\Schaffnop \ra \Spc
\end{equation}
preserves filtered colimits, i.e.\ sends (co-)filtered limits of affine schemes to filtered colimits of spaces.
\end{parag}

\begin{parag}
\label{subsubsec:laft-prestack}
When the base is Noetherian we refer to condition \ref{subsubsec:lafp-prestack} as \emph{locally almost of finite type}.
\end{parag}

\begin{parag}
In particular, we say that $\classical{\sX}:= {^{\leq 0}\sX}$ is \emph{locally of finite presentation (or type)} if the functor
\[
    \classical{\sX}: \clSchaffop := ({^{\leq 0}\Schaff})^{\rm op} \ra \Spc
\]
preserves filtered colimits.
\end{parag}

\begin{parag}
A ring $B$ is a \emph{complete Noetherian local ring} if
\begin{defnlist}
    \item(Noetherian) $B$ is a Noetherian ring; 
    \item(local) $H^0(B)$ is a local ring with maximal ideal $\fm_0$;
    \item(classically complete)\footnote{Notice that by \cite[Remark 7.3.6.7]{SAG} this is equivalent to requiring that $B$ is $\fm_0$-adically complete in the sense of \cite[Definition 7.3.1.1]{SAG}} $H^0(B)$ is $\fm_0$-adically complete, i.e.\ 
    \[
        H^0(B) \overset{\simeq}{\ra} \lim_{n \geq 1}H^0(B)/\fm_0^n.
    \]
\end{defnlist}
\end{parag}

\begin{parag}
A prestack $\sX$ is said to be \emph{integrable} if for any $B$ a complete Noetherian local ring with maximal ideal $\fm$\footnote{Here $\fm \subset B$ is a lift the maximal ideal $\fm_0 \subset H^0(B)$.} the canonical morphism
\[
    \sX(\Spec\,B) \overset{\simeq}{\ra} \lim_{n \geq 1}\sX(\Spec(B/\fm^n));
\]
is an isomorphism.
\end{parag}

\begin{parag}
A ring $A$ is said to be \emph{Artinian} if
\begin{defnlist}
    \item $H^0(A)$ is Artinian;
    \item $H^i(A)$ is a finitely generated $H^0(A)$-module, for all $i \in \bZ$;
    \item $H^i(A) = 0$ for $i \ll 0$. 
\end{defnlist}
An Artinian ring $A$ is \emph{local} if $H^0(A)$ is local.
\end{parag}

\begin{parag}
We say that a prestack $\sX$ is \emph{convergent} if for every affine scheme $S$ one has an equivalence
\[
    \sX(S) \ra \lim_{n \geq 0}\sX(\tau^{\leq n}(S)).
\]
\end{parag}

\begin{parag}
\label{subsubsec:truncated-prestack}
Given a prestack $\sX$ we denote by $\classical{\sX}:\clSchaffop \ra \Spc$ its truncation to classical affine schemes. Moreover, we say that $\classical{\sX}$ is \emph{$n$-truncated} if for every classical affine scheme $T_0$ one has $\classical{\sX}(T_0)$ is an $n$-truncated space.
\end{parag}

\begin{parag}
A space $X$ is $n$-truncated if 
\[
    \pi_k(X) = 0
\]
for all $k > n$.
\end{parag}

\subsection{Variants}

In this section we list many variants on the \hyperref[thm:main-result]{Main Theorem}.

Theorem \ref{thm:representability-natural} is essentially the version proved in \cite[Theorem 7.1.6]{DAG} and in our opinion is the most ``natural'' statement of the result.

\begin{thmannounce}
\label{thm:representability-natural}
Consider a map $p:\sX \ra S$, where $\sX$ is a prestack and $S$ is a Grothendieck affine scheme, then
\[
    \sX\mbox{ is an }n\mbox{-geometric stack locally almost of finite presentation}
\]
if and only if the following conditions are satisfied:
    \begin{enumerate}
    \item $\sX$ is a sheaf with respect to the \'etale topology;
    \item $\sX$ is locally almost of finite presentation;
    \item $\sX$ is integrable;
    \item[(iv)'\textsubscript{all}] $\sX$ is infinitesimally cohesive;
    \item[(v)\textsubscript{all}] $\sX$ admits a cotangent complex;
    \item[(vi)] $\sX$ is convergent;
    \item[$(vii)_n$] the underlying classical prestack $\classical{\sX}$ is $n$-truncated.
\end{enumerate}
\end{thmannounce}

\begin{parag}
\label{subsubsec:Main-Thm-and-Thm-A}
Theorem \ref{thm:representability-natural} follows from Theorem \ref{thm:main-result}. If $\sX$ is $n$-geometric then it admits corepresentable deformation theory by Proposition \ref{prop:n-geometric-has-deformation-theory}, so it satisfies conditions (iv)'\textsubscript{all}, (v)\textsubscript{all} and (vi) from Theorem \ref{thm:representability-natural}. Conversely, condition (iv)'\textsubscript{all}, (v)\textsubscript{all} and (vi) clearly imply condition (iv)\textsubscript{triv}, see Proposition \ref{prop:equivalent-formally-cohesive}; so the sufficiency of the conditions of Theorem \ref{thm:representability-natural} follows from those of Theorem \ref{thm:main-result}.
\end{parag}

\begin{rem}
\label{rem:integrable-from-classical-integrable}
In \cite{DAG} condition (iii) in Theorem \ref{thm:representability-natural} is substituted by the weaker: 
\begin{enumerate}
    \item[$(iii)_{\rm c\ell}$] $\classical{\sX}$ is integrable, i.e.\ for any local complete discrete Noetherian ring $B_0$ with maximal ideal $\fm_0 \subset B_0$ the canonical map
    \[
        \classical{\sX}(\Spec\,B_0) \overset{\simeq}{\ra} \lim_{n \geq 1}\classical{\sX}(\Spec\,B_0/\fm^n_0)
    \]
    is an equivalence.
\end{enumerate}
When $\sX$ admits deformation theory, condition (iii)\textsubscript{$\rm c\ell$} implies (iii) by \cite[Proposition 7.1.7]{DAG}.
\end{rem}

\begin{rem}
\label{rem:natural-plus-diagonal}
In fact, if $\sX$ satisfies the conditions of Theorem \ref{thm:representability-natural} plus the extra condition 
\begin{enumerate}
    \item[(viii)] the diagonal morphism $\sX \ra \sX \underset{S}{\times}\sX$ is geometric,
\end{enumerate}
then we can directly apply Theorem \ref{thm:etale-surjection} to obtain that $\sX$ is geometric.
\end{rem}

\begin{parag}
Theorem \ref{thm:representability-laft-cotangent-complex} helps simplify condition (ii) by restricting it to (ii)\textsubscript{$\rm c\ell$}, which is easier to verify since it only depends on the data of $\classical{\sX}$.
\end{parag}

\begin{thmannounce}
\label{thm:representability-laft-cotangent-complex}
Consider a map $p:\sX \ra S$, where $\sX$ is a prestack and $S$ is a Grothendieck affine scheme, then
\[
    \sX\mbox{ is an }n\mbox{-geometric stack locally almost of finite presentation}
\]
if and only if the following conditions are satisfied:
    \begin{enumerate}
    \item $\sX$ is a sheaf with respect to the \'etale topology;
    \item[$(ii)_{\rm c\ell}$] $\classical{\sX}$ is locally of finite presentation;
    \item[(iii)] $\sX$ is integrable;
    \item[$(iv)'_{\rm all}$] $\sX$ is infinitesimally cohesive;
    \item[$(v)^{\rm laft}_{\rm all}$]\footnote{Here the notation for the subcript indicates which affine points we require the cotangent complex to exists at, whereas the superscript indicates that we require that on \emph{eventually coconnective affine schemes of finite type} we require that the cotangent complex is a locally almost of finite type object. One could also have considered variants of this condition where we are required to check the laft property of the cotangent complex on all points, but these are not relevant for us in this work.} $\sX$ admits a cotangent complex and for every $(T_0\overset{x_0}{\ra}\sX) \in (\clSchaffft)_{/\sX}$ one has
    \[
        T^*_{x_0}(\sX) \in \Coh(T_0)^-;
    \]
    \item[(vi)] $\sX$ is convergent;
    \item[$(vii)_n$] the underlying classical prestack $\classical{\sX}$ is $n$-truncated.
\end{enumerate}
\end{thmannounce}

\begin{rem}
By the first part of the proof of Lemma \ref{lem:cotangent-complex-from-ft-points}, we can substitute condition (v)$^{\rm laft}_{\rm all}$ by:
\begin{enumerate}
    \item[$(v)^{\rm laft}_{\rm ec}$] for every $(T \overset{x}{\ra} \sX) \in (\Schaffconv)_{/\sX}$, $\sX$ admits a cotangent complex at $x$ and for every $(T_0\overset{x_0}{\ra}\sX) \in (\clSchaffft)_{/\sX}$ one has
    \[
        T^*_{x_0}(\sX) \in \Coh(T_0)^-.
    \]
\end{enumerate}
\end{rem}

\begin{parag}
\label{subsubsec:Thm-A-implies-Thm-B}
Theorem \ref{thm:representability-natural} implies Theorem \ref{thm:representability-laft-cotangent-complex}. First notice that if $\sX$ is a locally almost of finite type and $n$-geometric stack, by \cite[Lemma 3.5.2]{GRII} for every $(T_0\overset{x}{\ra}\sX) \in (\clSchaffft)_{/\sX}$ a classical affine scheme of finite type one has\footnote{See \cite[Chapter 1, \S 3.4.1]{GRII} for a definition of the category $\Pro(\QCoh(T_0)^-)_{\rm laft}$, we just notice that when $T$ is an eventually coconnective affine scheme one has an equivalence $\Pro(\QCoh(T)^-)_{\rm laft} \simeq \Pro(\Coh(T))$.}
\[
    T^*_{x_0}(\sX) \in \QCoh(T_0)^{-} \cap \Pro(\QCoh(T_0)^-)_{\rm laft} \simeq \QCoh(T_0)^{-} \cap \Pro(\Coh(T_0)) \simeq \Coh(T_0)^-.
\]
Now, given $\sX$ satisfying the conditions of Theorem \ref{thm:representability-laft-cotangent-complex}, by \cite[Chapter 1, Theorem 9.1.2]{GRII} $\sX$ is locally almost of finite presentation, hence satisfies condition (ii) of Theorem \ref{thm:representability-natural}.
\end{parag}

\begin{rem}
\label{rem:descent-from-classical-descent}
Notice that we can simplify Theorem \ref{thm:representability-laft-cotangent-complex} further by instead of (i) requiring the weaker:
\begin{enumerate}
    \item[$(i)_{\rm c\ell}$] $\classical{\sX}$ is an \'etale sheaf.
\end{enumerate}
Indeed, by Remark \ref{rem:n-truncated-and--n-connective} we know that the cotangent complex of $\sX$ is $(-n)$-connective, thus we can apply \cite[Chapter 1, Proposition 8.2.2]{GRII} which gives that $\sX$ satisfy \'etale descent.
\end{rem}

\begin{parag}
The last variant that we formulate tries to simplify how one checks condition ${\rm (v)}^{\rm laft}_{\rm all}$ in Theorem \ref{thm:representability-laft-cotangent-complex}.
\end{parag}

\begin{thmannounce}
\label{thm:cohomological-cotangent-complex-condition}
Consider a map $p:\sX \ra S$, where $\sX$ is a prestack and $S$ is a Grothendieck affine scheme, then
\[
    \sX\mbox{ is an }n\mbox{-geometric stack locally almost of finite presentation}
\]
if and only if the following conditions are satisfied:
    \begin{enumerate}
    \item[$(i)_{\rm c\ell}$] $\classical{\sX}$ is a sheaf with respect to the \'etale topology;
    \item[$(ii)_{\rm c\ell}$] $\classical{\sX}$ is locally of finite presentation;
    \item[(iii)] $\sX$ is integrable;
    \item[$(iv)_{\rm global}$] $\sX$ is $W_{\rm all}$-homogeneous;
    \item[$(\alpha)$] for every affine scheme $(S \overset{x}{\ra}\sX) \in \Schaff_{/\sX}$ the functor\footnote{We distiguished this condition from $(v)_{\rm all}$ which also require the existence of $T^*_x\sX$, here the existence is guaranteed by $(iv)_{\rm global}$.}
    \[
        \Map_{\Pro(\QCoh(S)^-)}(T^*_x\sX,-):\QCoh(S)^{\leq 0} \ra \Spc
    \]
    preserves co-filtered limits\footnote{Notice that this not automatic, since for $\sF \simeq \lim_{J}\sF_j \in \QCoh(S)^{\leq 0}$ one has the following formula
    \[
        \Map_{\Pro(\QCoh(S)^-)}(T^*_x\sX,\sF) := \colim_{I^{\rm op}}\Map_{\QCoh(S)^{-}}(\sG_i,\sF),
    \]
    where $T^*_x\sX \simeq \lim_I\sG_i$ is some presentation of $T^*_{x}\sX$ by objects $\sG_i \in \QCoh(S)^-$ and a co-filtered diagram $I$.
    };
    \item[($\beta$)] for every $x_0:T_0\ra \classical{\sX}$ where $S_0$ is a classical affine scheme of finite type one has
    \[
        H^k(T^*_{x_0}(\sX)) \in \Coh(T_0)^{\heartsuit}
    \]
    for all $k \in \bZ$.
    \item[(vi)] $\sX$ is convergent;
    \item[$(vii)_n$] the underlying classical prestack $\classical{\sX}$ is $n$-truncated.
\end{enumerate}
\end{thmannounce}

\begin{parag}
\label{subsubsec:cohomological-from-laft}
Notice that Proposition \ref{prop:affine-homogeneous-is-deformation-theory} gives that $(\alpha)$ is equivalent to $\sX$ admitting deformation theory, that is $\sX$ satisfies (iv)'\textsubscript{all} and $\sX$ admits a \emph{pro}-cotangent complex, i.e.\ the weaker form of (v)\textsubscript{all}: 
\begin{enumerate}
    \item[$(v)^{\rm Pro}_{\rm all}$] $\sX$ admits a pro-cotangent complex at every $(S\overset{x}{\ra}\sX) \in \Schaff_{/\sX}$;
\end{enumerate}
In particular, given $x_0:S_0 \ra \sX$ a morphism from $S_0 \in \clSchaffft$ a classical affine scheme of finite type, one has
\[
    T^*_{x_0}(\sX) \in \Pro(\Coh(S_0)^-).
\]
Now by Lemma \ref{lem:coherent-cohomology-implies-coherent} conditions $(\alpha)$ and $(\beta)$ imply that each $T^*_{x_0}(\sX) \in \Coh(S_0)^-$, so that $\sX$ satisfies $(v)^{\rm laft}_{\rm all}$. Moreover, since $\sX$ admits deformation theory, Remark \ref{rem:descent-from-classical-descent} gives that $\sX$ satisfies \'etale descent. So we obtained the conditions from Theorem \ref{thm:representability-laft-cotangent-complex}.
\end{parag}

\begin{rem}
We notice that Theorem \ref{thm:cohomological-cotangent-complex-condition} is closely related to the results \cite[Theorem 7.5.1]{DAG} and \cite[Corollary 1.36]{MR3004173}. Moreover, in the formulation of Theorem \ref{thm:cohomological-cotangent-complex-condition} we can conceive of computing more directly defined cohomology groups (cf.\ \cite[Definition 1.13]{MR0277519} or \cite[\S 7.4]{DAG}) whose finiteness would guarantee the finiteness encoded in condition $(\beta)$. However, it seems to us that it is unlikely that in practice these other cohomology groups are easier to compute than $H^i(T^*_{x_0}\sX)$ given that the later has better formal properties when compared to the former.
\end{rem}

\begin{parag}
Figure \ref{fig:dependency-scheme} summarize all the implications involved in the different variants discussed in this section. We clarify that ``def'' (resp.\ ``corep-def'') stands for (resp.\ corepresentable) deformation theory and ``laft-def'' (resp.\ ``laft-corep-def'') stands for (resp.\ corepresentable) deformation theory satisfying the condition that the pro-cotangent (resp.\ cotangent) complex is a locally almost of finite type object.
\end{parag}

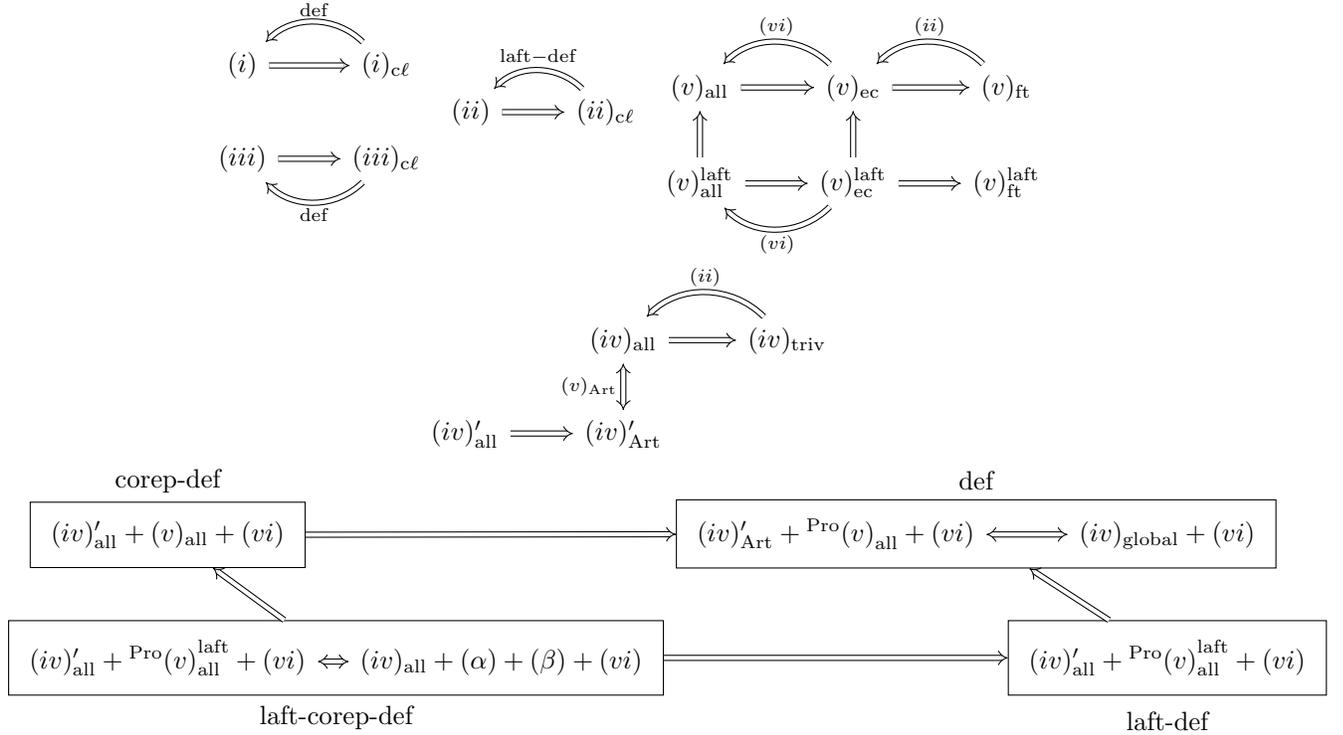
\begin{figure}[!h]
    \centering
    \begin{tikzcd}
    (i) \ar[Rightarrow]{r} & (i)_{\rm c\ell} \ar[Rightarrow,bend right = 45,"{\rm def}"']{l} \\
    (iii) \ar[Rightarrow]{r} & (iii)_{\rm c\ell} \ar[Rightarrow,bend left = 45,"{\rm def}"]{l} \\
    \end{tikzcd}
    \begin{tikzcd}
    (ii) \ar[Rightarrow]{r} & (ii)_{\rm c\ell} \ar[Rightarrow,bend right = 45,"{\rm laft-def}"']{l} \\
    \end{tikzcd}
    \begin{tikzcd}
    (v)_{\rm all} \ar[Rightarrow]{r} & (v)_{\rm ec} \ar[Rightarrow]{r} \ar[Rightarrow,bend right = 45,"(vi)"']{l} & (v)_{\rm ft} \ar[Rightarrow,bend right = 45,"(ii)"']{l} \\
    (v)^{\rm laft}_{\rm all} \ar[Rightarrow]{r} \ar[Rightarrow]{u} & (v)^{\rm laft}_{\rm ec} \ar[Rightarrow,bend left = 45,"(vi)"]{l} \ar[Rightarrow]{u} \ar[Rightarrow]{r} & (v)^{\rm laft}_{\rm ft}
    \end{tikzcd}
    \begin{tikzcd}
    & (iv)_{\rm all} \ar[Rightarrow]{r} & (iv)_{\rm triv} \ar[Rightarrow,bend right = 45,"(ii)"']{l} \\
    (iv)'_{\rm all} \ar[Rightarrow]{r} & (iv)'_{\rm Art} \ar[Leftrightarrow,"(v)_{\rm Art}"]{u} & 
    \end{tikzcd}
    \begin{tikzcd}[
      column sep=small, row sep=small,
    /tikz/execute at end picture={
    \node (corep-def) [rectangle, draw, fit=(A1),label={corep-def}] {};
    \node (def) [rectangle, draw, fit=(A2) (A3),"def"] {};
    \node (laft-corep-def) [rectangle, draw, fit=(B1) (B2),label={[yshift=-45pt]:laft-corep-def}] {};
    \node (laft-def) [rectangle, draw, fit=(B3),label={[yshift=-45pt]:laft-def}] {};
    \draw[->,double equal sign distance,-Implies] (corep-def) -- (def){};
    \draw[->,double equal sign distance,-Implies] (laft-corep-def) -- (corep-def){};
    \draw[->,double equal sign distance,-Implies] (laft-corep-def) -- (laft-def){};
    \draw[->,double equal sign distance,-Implies] (laft-def) -- (def){};
    }]
    |[alias=A1]| (iv)'_{\rm all} + (v)_{\rm all} + (vi) & & |[alias=A2]| (iv)'_{\rm Art} + {^{\rm Pro}(v)}_{\rm all} + (vi) \arrow[Leftrightarrow]{r} & |[alias=A3]| (iv)_{\rm global} + (vi) \\
    & & & \\
    & & & \\
    |[alias=B1]| (iv)'_{\rm all} + {^{\rm Pro}(v)}^{\rm laft}_{\rm all} + (vi) \arrow[Leftrightarrow]{r} & |[alias=B2]| (iv)_{\rm all}+(\alpha)+(\beta) +(vi) & & |[alias=B3]| (iv)'_{\rm all} + {^{\rm Pro}(v)}^{\rm laft}_{\rm all} + (vi) \\
    \end{tikzcd}
    \caption{Summary of dependency of conditions.}
    \label{fig:dependency-scheme}
\end{figure}

\subsection{Future questions}

\subsubsection{}

Probably the most interesting question to be understood at the moment is: 
\begin{question}
Is there a version of the \hyperref[thm:main-result]{Main Theorem} for $n$-geometric stacks over an arbitrary base, possibly even non-connective, where one considers the notion of fiber smooth in the definition of $n$-geometric stacks?
\end{question}

The author believes that a version of Theorem \ref{thm:representability-natural} to appear in Chapter 27 of \cite{SAG} will address this question.

\subsubsection{}

One of the initial motivations for this note was to relax as much as possible condition (v)\textsubscript{ft} in \hyperref[thm:main-result]{Main Theorem}. The bootstrapping result of Lemma 10.4 in \cite{Hall-Rydh-algebraicity} suggests that one could try to relax condition (v)\textsubscript{ft} to something as condition (v)\textsubscript{Art} in Corollary \ref{cor:formal-charts}. Unfortunately, with the current argument using approximate charts (Proposition \ref{prop:approximate-chart}) we could not relax this assumption. However, we still pose

\begin{question}
\label{que:relax-condition-(v)}
Can one relax condition (v)\textsubscript{ft} in \hyperref[thm:main-result]{Main Theorem}, so that we only need to check the existence the cotangent complex at field valued points?
\end{question}

\subsubsection{}

In checking the conditions to apply any form of the representability theorem we want to stress the distinction between knowing that the cotangent complex exists and computing it. In Theorem \ref{thm:main-result} and Theorem \ref{thm:representability-natural} we only need to know that a cotangent complex exists, whereas in Theorem \ref{thm:representability-laft-cotangent-complex} and Theorem \ref{thm:cohomological-cotangent-complex-condition} we actually need to know that the cotangent complex has a certain finiteness property. Condition (iv)\textsubscript{global} provides a good way to know that a pro-cotangent complex exists and condition ($\beta$) is simple enough to check when a cotangent space is actually locally almost of finite type. However we would like to answer:

\begin{question}
What is the simplest condition that we can impose instead of $(\alpha)$ above that together with (iv)\textsubscript{global} and (vi) guarantee the existence of \emph{corepresentable} deformation theory?
\end{question}

Theorem 17.5.4.1 of \cite{SAG} suggests that in the case where the base ring $R$ admits a dualizing complex, one can probably obtain $(\alpha)$ from some base change compatibility of the tangent spaces.

\subsubsection{}

Our theorem does not imply any version of the Main Theorem in \cite{Hall-Rydh-algebraicity}. Thus, it is natural to ask:

\begin{question}
Is there a version of Theorem \ref{thm:main-result} which implies the result of \cite{Hall-Rydh-algebraicity} when restricted to $1$-truncated classical prestacks?
\end{question}

\subsubsection{}
In \emph{loc. cit.} one has a DVR-homogeneity condition which plays a role in proving their version of Lemma \ref{lem:main-lemma} that allows to bootstrap homogeneity once the diagonal is representable. We wonder if a form of this condition in the derived setting could help to address Question \ref{que:relax-condition-(v)} above.

\subsubsection{}

Another point of possible investigation is how far one can simplify how to check the locally almost of finite type condition. Corollary 9.1.3 from \cite{GRII} gives a very strong tool to reduce condition (ii) to (ii)\textsubscript{$\rm c\ell$}, however one needs to know that $\sX$ has laft deformation theory, which then requires us to check the existence of the cotangent complex for affine schemes which are not of finite type. More succiently we would want to answer:

\begin{question}
Can we simplify condition (v)$^{\rm laft}_{\rm ec}$ in Theorem \ref{thm:representability-laft-cotangent-complex} to (v)$^{\rm laft}_{\rm ft}$?
\end{question}

\subsubsection{}
We notice that, probably, the answer to the above question is no when $\sX$ is not $n$-geometric, however if we further impose condition (viii) in Theorem \ref{thm:representability-laft-cotangent-complex} then by Lemma \ref{lem:main-lemma} we can relax (v)$^{\rm laft}_{\rm ec}$ to (v)$^{\rm laft}_{\rm ft}$.

\subsubsection{}

Finally, we mention that this whole note was written as a warm-up exercise in trying to proof a version of Theorem \ref{thm:representability-natural} for ind-schemes. We expect that by relaxing condition $(v)_{\rm all}$ to $(v)^{\rm Pro}_{\rm all}$ one can still prove that $\sX$ is a filtered colimit of $n$-geometric stacks. We are currently pursuing this question.

\subsection{Acknowledgements}

Our note clearly owns a huge debt to Lurie's result, since we essentially follow the same strategy to prove this version of the theorem, as well as use a lot of his foundational results. This work started from a desire to understand the extremely interesting work \cite{Hall-Rydh-algebraicity}, to which we also own a large amount. I would like to thank E.\ Elmanto for drawing the author's attention to the notes \cite{Alper-Mainz} that helped the author better understand the current status of the literature regarding the classical representability result. I would also like to thank X.\ He and M.\ McBreen for discussions on this subject and for giving me the opportunity to teach a seminar on derived algebraic geometry, which allowed me to better understand the representability result. Finally, I would like to thank N.\ Rozenblyum for conversations many years ago about this subject, including explaining how to obtain Theorem \ref{thm:representability-laft-cotangent-complex} from Theorem \ref{thm:representability-natural}.

\section{Derived analogues of classical notions}

\subsection{Homogeneity}

The discussion of this section is a straightforward extension of the concepts defined in \S 1 of \cite{Hall-Rydh-algebraicity} to the context of derived geometry.

\begin{defn}
    Given a morphism $f:T_1 \ra T_2$ of affine schemes we say that 
    \begin{defnlist}
        \item $f$ is a \emph{nilpotent embedding} if $\classical{f}:\classical{T_1} \ra \classical{T_2}$ is a closed embedding of classical affine schemes with nilpotent ideal of definition--let $W_{\rm nil.}$ denote the class of morphisms of affine schemes consisting of nilpotent embeddings;
        \item $f$ is a \emph{closed embedding} if $\classical{f}:\classical{T_1} \ra \classical{T_2}$ is a closed embedding of classical affine schemes--let $W_{\rm cl.}$ denote the class of morphisms of affine schemes consisting of nilpotent embeddings.
    \end{defnlist}
    We let $W_{\rm all}$ denote the class of all morphisms between affine schemes.
\end{defn}

\begin{parag}
Let $\SchafflA$ denote the subcategory of affine schemes generated by local Artinian rings $A$ of finite type, i.e.\ such that the structure morphism $\Spec(A) \ra S$ is almost of finite presentation. For morphisms in $\SchafflA$ we impose the following condition 
\end{parag}

\begin{defn}
    Given a morphism $f:U_1 \ra U_2$ of local Artinian affine schemes we say that $f$ is an \emph{inseparable} morphism if the induced morphism between closed points
    \[
        \bar{f}: u_1 \ra u_2
    \]
    is a purely inseparable field extension--let $W^{\rm lA}_{\rm insep.}$ denote the class of inseparable morphisms between local Artinian affine schemes of finite type.
    We let $W^{\rm lA}_{\rm all}$ denote the class of all local morphisms between local Artinian affine schemes.
\end{defn}

\begin{parag}
\label{subsubsec:P-classes}
Given a class $P \in \{W_{\rm nil.},W_{\rm all},W^{\rm lA}_{\rm triv.},W^{\rm lA}_{\rm all}\}$ a \emph{$P$-nil square} is a pushout square
\begin{equation}
    \label{eq:P-nil-square}
    \begin{tikzcd}
        T_1 \ar[r,"f"] \ar[d,hook,"\jmath_1"'] & T_2 \ar[d] \\
        T'_1 \ar[r] & T'_2
    \end{tikzcd}    
\end{equation}
where $f \in P$ and $\jmath_1$ is a nilpotent embedding. Notice that when we require that $f$ belongs to the subset $\{W^{\rm lA}_{\rm triv.},W^{\rm lA}_{\rm all}\}$ this means that $T_1$ and $T_2$ are local Artinian affine schemes and so is $T'_1$ since $\jmath_1$ is a nilpotent embedding.
\end{parag}

\begin{defn}
\label{defn:P-homogeneous}
Given a prestack $\sX$ and a class $P$ as in \S \ref{subsubsec:P-classes} we say that $\sX$ is \emph{$P$-homogeneous} if for every square as (\ref{eq:P-nil-square}) the following natural map
\[
    \sX(T'_2) \ra \sX(T'_1) \underset{\sX(T_1)}{\times}\sX(T_2)
\]
is an isomorphism. 

More generally, we say that a morphism of prestacks $g:\sX \ra \sY$ is \emph{$P$-homogeneous}\footnote{In the particular case of the morphisms $\{W^{\rm lA}_{\rm triv.}, W^{\rm lA}_{\rm all}\}$ we will require that $T_1 \ra \sY$ and $T_2 \ra \sY$ are locally almost of finite presentation. This condition is automatic when $\sY$ is locall almost of finite presentation.} if for every square as (\ref{eq:P-nil-square}) the following is a pullback square
\[
    \begin{tikzcd}
    \sX(T'_2) \ar[r] \ar[d] & \sX(T'_1) \underset{\sX(T_1)}{\times}\sX(T_2) \ar[d] \\
    \sY(T'_2) \ar[r] & \sY(T'_1) \underset{\sY(T_1)}{\times}\sY(T_2).
    \end{tikzcd}
\]
\end{defn}

\begin{rem}
Notice that the condition that $\sX$ is $W_{\rm nil.}$-homogeneous is exactly the same as $\sX$ being infinitesimally cohesive as defined in Definition \ref{defn:infinitesimally-cohesive}.
\end{rem}

\begin{rem}
Notice that in analogy with \cite[Tag 06L9]{stacks-project} we could have formulated the following: we say that a prestack $\sX$ satisfies the {\em (derived) Rim--Schlessinger condition} if $\sX$ takes any pushout diagram as follows
\[
    \begin{tikzcd}
        T_1 \ar[r,"f"] \ar[d,hook,"\jmath_1"'] & T_2 \ar[d] \\
        T'_1 \ar[r] & T'_2
    \end{tikzcd}    
\]
to a pullback diagram, where $T_1,T_2,T'_1$ and $T'_2$ are (locally almost of finite presentation\footnote{Again this condition is only relevant if one is writing a relative version of the Rim--Schlessinger condition, e.g.\ one considers $\sX \ra S$ a prestack over some (derived) scheme $S$.}) local Artinian schemes, $f$ is arbitrary and $\jmath_1$ is a closed embedding. Since any closed embedding of local Artinian schemes is nilpotent this condition is equivalent to $W^{\rm lA}_{\rm all}$-homogeneity.
\end{rem}

For the purposes of bootstrapping homogeneity we need to consider a refinement of the condition above:
\begin{defn}
\label{defn:H1P-and-H2P-conditions}
    A prestack $\sX$ is said to satisfy
    \begin{enumerate}
        \item[$(H^{\rm P}_1)$] if for every square as (\ref{eq:P-nil-square}) the natural map
        \[
            \sX(T'_2) \ra \sX(T'_1) \underset{\sX(T_1)}{\times}\sX(T_2)
        \]
        is fully faithful;
        \item[$(H^{\rm P}_2)$] if for every square as (\ref{eq:P-nil-square}) the natural map
        \[
            \sX(T'_2) \ra \sX(T'_1) \underset{\sX(T_1)}{\times}\sX(T_2)
        \]
        is essentially surjective.
    \end{enumerate}
    Clearly, $\sX$ is $P$-homogeneous if and only if it satisfies $H^{\rm P}_1$ and $H^{\rm P}_2$.
    
    More generally, given a morphism of prestacks $g:\sX \ra \sY$ then $g$ satisfies $(H^{\rm P}_1)$ (resp.\ $(H^{\rm P}_2)$) if the canonical map
    \[
      \sX(T'_2) \ra \sY(T'_2) \underset{\left(\sY(T'_1) \underset{\sY(T_1)}{\times}\sY(T_2)\right)}{\times} \left(\sX(T'_1) \underset{\sX(T_1)}{\times}\sX(T_2)\right)
    \]
    is fully faithful (resp.\ essentially surjective).
\end{defn}

We recall the following formal result, which will be useful when checking the conditions in Definition \ref{defn:H1P-and-H2P-conditions}.

\begin{lem}
\label{lem:categorical-conditions-for-spaces}
Let $F:S \ra T$ be a morphism in the category $\Spc$ then
    \begin{enumerate}
        \item $F$ is essentially surjective (as a morphism of $\infty$-categories) if and only if
        \[
            \pi_0(F): \pi_0(S) \ra \pi_0(T) \;\;\; \mbox{ is surjective.}
        \]
        \item $F$ is fully faithful (as a morphism of $\infty$-categories) if and only if
        \[
            \pi_i(F): \pi_i(S,s) \ra \pi_i(T,F(s))
        \]
        is bijective, for every $s \in \pi_0(S)$ and $i \geq 1$.
    \end{enumerate}
\end{lem}

We check the compatibilities between the relative and absolute notion.

\begin{lem}
\label{lem:compatible-relative-absolute}
Let $f:\sX \ra \sY$ be a morphism of prestacks. The following conditions are equivalent:
\begin{resultlist}
    \item $f$ satisfies $(H^{\rm P}_1)$ (resp.\ $(H^{\rm P}_2)$);
    \item for every affine scheme $U \ra \sY$ 
    \[
        \sX\underset{\sY}{\times}U
    \]
    satisfies $(H^{\rm P}_1)$ (resp.\ $(H^{\rm P}_2)$).
\end{resultlist}
\end{lem}

\begin{proof}
We prove the case of $(H^{\rm P}_1)$, the other is completely analogous. 

Consider a diagram as in (\ref{eq:P-nil-square}) and a morphism $U \ra \sY$ from an affine scheme $U$. Notice we have a pullback diagram
\begin{equation}
    \label{eq:pullback-of-canonical-morphism}
    \begin{tikzcd}
    \sX(T'_2)\underset{\sY(T'_2)}{\times}U(T'_2) \ar[r] \ar[d] & \left(\sX(T'_1)\underset{\sX(T_1)}{\times}\sX(T_2) \underset{\sY(T'_1)\underset{\sY(T_1)}{\times}\sY(T_2)}{\times}\sY(T'_2)\right)\underset{\sY(T'_2)}{\times}U(T'_2) \ar[d] \\
    \sX(T'_2) \ar[r] & \sX(T'_1)\underset{\sX(T_1)}{\times}\sX(T_2) \underset{\sY(T'_1)\underset{\sY(T_1)}{\times}\sY(T_2)}{\times}\sY(T'_2)
    \end{tikzcd}
\end{equation}
and if the lower horizontal morphism is fully faithful then so is the upper horizontal morphism.

Since pullbacks of prestacks are computed point-wise the upper horizontal map in (\ref{eq:pullback-of-canonical-morphism}) becomes
\begin{equation}
    \label{eq:iterated-fiber-product-canonical-map}
    (\sX\underset{\sY}{\times}U)(T'_2) \ra \left(\sX(T'_1)\underset{\sX(T_1)}{\times}\sX(T_2) \underset{\sY(T'_1)\underset{\sY(T_1)}{\times}\sY(T_2)}{\times}\sY(T'_2)\right)\underset{\sY(T'_2)}{\times}U(T'_2).
\end{equation}
Notice, that since any affine scheme is $P$-homogeneous, one has that the right-hand side of (\ref{eq:iterated-fiber-product-canonical-map}) is isomorphic to
\begin{equation}
\label{eq:pullback-at-affine-then-at-prestacks}
    \left(\sX(T'_1)\underset{\sX(T_1)}{\times}\sX(T_2)\right) \underset{\sY(T'_1)\underset{\sY(T_1)}{\times}\sY(T_2)}{\times} \left(U(T'_1)\underset{U(T_1)}{\times}U(T_2)\right).
\end{equation}
However, (\ref{eq:pullback-at-affine-then-at-prestacks}) is equivalent to the following iterated pullback
\begin{equation}
    \label{eq:pullback-at-prestacks-then-at-affine}
    \left(\sX(T'_1)\underset{\sY(T'_1)}{\times}U(T'_1)\right)\underset{\sX(T_1)\underset{\sY(T_1)}{\times}U(T_1)}{\times}\left(\sX(T_2)\underset{\sY(T_2)}{\times}U(T_2)\right).
\end{equation}
Indeed, in considering the big diagram
\[
    \begin{tikzcd}
    \sX(T'_1) \ar[r] \ar[d] & \sY(T'_1) \ar[d] & U(T'_1) \ar[d] \ar[l] \\
    \sX(T_1) \ar[r] & \sY(T_1) & U(T_1) \ar[l] \\
    \sX(T_2) \ar[r] \ar[u] & \sY(T_2) \ar[u] & U(T_2) \ar[l] \ar[u]
    \end{tikzcd}
\]
one notices that (\ref{eq:pullback-at-affine-then-at-prestacks}) is the pullback at each column then at each resulting row, whereas (\ref{eq:pullback-at-prestacks-then-at-affine}) is the pullback at each row then at each resulting column.

Thus, the map (\ref{eq:iterated-fiber-product-canonical-map}) becomes
\begin{equation}
    \label{eq:canonical-map-transformed}
    (\sX\underset{\sY}{\times}U)(T'_2) \ra (\sX\underset{\sY}{\times}U)(T'_1)\underset{(\sX\underset{\sY}{\times}U)(T_1)}{\times}(\sX\underset{\sY}{\times}U)(T_2).
\end{equation}

Now, if (a) is satisfied, since fully faithful morphisms are stable under pullback from (\ref{eq:canonical-map-transformed}) one sees that (b) is also satisfied. 

Conversely, assume that (b) holds and take $U = T'_2$, then the upper horizontal arrow of (\ref{eq:pullback-of-canonical-morphism}) becomes
\[
    \sX(T'_2)\underset{\sY(T'_2)}{\times}\pt \overset{\alpha}{\ra} \left(\sX(T'_1)\underset{\sX(T_1)}{\times}\sX(T_2) \underset{\sY(T'_1)\underset{\sY(T_1)}{\times}\sY(T_2)}{\times}\sY(T'_2)\right)\underset{\sY(T'_2)}{\times}\pt.
\]

Thus, one obtains a fiber sequence
\[
    \Fib(\pt \ra \sY(T'_2)) \ra \Fib(\alpha) \ra \Fib(\beta),
\]
where $\beta$ is the lower horizontal morphism in (\ref{eq:pullback-of-canonical-morphism}). Notice that by Lemma \ref{lem:categorical-conditions-for-spaces} $\alpha$ is fully faithful if and only if it is $(-1)$-truncated, since $\Fib(\pt \ra \sY(T'_2))$ and $\Fib(\alpha)$ are $(-1)$-truncated, so is $\Fib(\beta)$, which finishes the proof.
\end{proof}

\begin{rem}
Notice that Lemma \ref{lem:compatible-relative-absolute} tautologically implies that a prestack $\sX$ satisfies $H^{\rm P}_1$ (resp.\ $H^{\rm P}_2$) if and only if the structure morphism $\sX\ra S$ satisfies $H^{\rm P}_1$ (resp.\ $H^{\rm P}_2$). More generally, one can check that if $f:\sX \ra \sY$ satisfies $H^{\rm P}_1$ (resp.\ $H^{\rm P}_2$) and $\sY$ satisfies $H^{\rm P}_1$ (resp.\ $H^{\rm P}_2$), then so does $\sX$.
\end{rem}

\subsection{Cohesiveness}

The notion of cohesiveness is very similar to homogeneity, except that it looks slightly more symmetric, i.e.\ both morphisms in the pushout diagrams are assumed to be closed or nilpotent embeddings.

We will not use this notion in this paper, we included it here just for completeness.

\begin{defn}
\label{defn:cohesive}
A prestack $\sX$ is said to be \emph{cohesive} if it takes any pushout square
\begin{equation}
    \label{eq:cohesive-defn-square}
    \begin{tikzcd}
    T_0 \ar[r,"\imath_{01}"] \ar[d,"\imath_{02}"'] & T_1 \ar[d] \\
    T_2 \ar[r] & T_{3}
    \end{tikzcd}
\end{equation}
where $\imath_{01}$ and $\imath_{02}$ are closed embeddings to a pullback square of spaces.

More generally, a morphism $g:\sX \ra \sY$ of prestacks is \emph{cohesive} if given any pushout square as in (\ref{eq:cohesive-defn-square}) the canonical map
\[
    \sX(T_3) \ra \sY(T_3) \underset{\left(\sY(T_2) \underset{\sY(T_0)}{\times}\sY(T_1)\right)}{\times} \left(\sX(T_2) \underset{\sX(T_0)}{\times}\sX(T_1)\right)
\]
is an isomorphism.
\end{defn}

\begin{rem}
The notion of cohesiveness will be satisfied by any geometric stack (see \cite[Theorem 5.6.4]{DAG}), it is however a bit too strong for bootstrapping deformation theory for an, a priori, arbitrary prestack.
\end{rem}

It turns out that in checking the axioms for a geometric stack the following weaker version of Definition \ref{defn:cohesive} is more appropriate.

\begin{defn}
\label{defn:infinitesimally-cohesive}
A prestack $\sX$ is said to be \emph{infinitesimally cohesive} if it takes any pushout square
\begin{equation}
    \label{eq:infinitesimally-cohesive-defn-square}
    \begin{tikzcd}
    T_0 \ar[r,"\jmath_{01}"] \ar[d,"\jmath_{02}"'] & T_1 \ar[d] \\
    T_2 \ar[r] & T_{3}
    \end{tikzcd}
\end{equation}
where $\jmath_{01}$ and $\jmath_{02}$ are nilpotent embeddings to a pullback square of spaces.

More generally, a morphism $g:\sX \ra \sY$ of prestacks is \emph{infinitesimally cohesive} if given any pushout square as in (\ref{eq:infinitesimally-cohesive-defn-square}) the canonical map
\[
    \sX(T_3) \ra \sY(T_3) \underset{\left(\sY(T_2) \underset{\sY(T_0)}{\times}\sY(T_1)\right)}{\times} \left(\sX(T_2) \underset{\sX(T_0)}{\times}\sX(T_1)\right)
\]
is an isomorphism.
\end{defn}

\begin{rem}
\label{rem:infinitesimally-cohesive-GR}
In \cite[Chapter 1, \S 6.1]{GRII} infinitesimally cohesive is defined by considering only squares as (\ref{eq:infinitesimally-cohesive-defn-square}) where $T_1 = T_2$ and $\jmath_{01} (= \jmath_{02})$ is a square-zero extension.
Nevertheless, when $\sX$ admits a pro-cotangent complex and is convergent this is equivalent to the definition given above by \cite[Proposition 17.3.6.1]{SAG}.
\end{rem}

This notion is closely related to $W^{\rm lA}_{\rm all}$ as we will see in Proposition \ref{prop:equivalent-formally-cohesive}.

\begin{defn}
\label{defn:formally-cohesive}
Given a prestack $\sX$ we will say that $\sX$ is \emph{formally cohesive} if it takes every pushout square
\begin{equation}
    \label{eq:pushout-square-formally-cohesive}
    \begin{tikzcd}
    U_0 \ar[r,"\imath_{01}"] \ar[d,"\imath_{02}"'] & U_1 \ar[d] \\
    U_2 \ar[r] & U_{3}
    \end{tikzcd}
\end{equation}
where $U_0,U_1,U_2$ and $U_3$ are local Artinian schemes of finite type, and $\imath_{01}$ and $\imath_{02}$ are closed embeddings to a pullback diagram of spaces.
\end{defn}

\subsection{Cotangent complex}

In this section we review the formalism of the cotangent complex for a prestack following \cite[Chapter 1]{GRII}. We refer the reader to \emph{loc. cit.} for details.

\begin{defn}
\label{defn:pro-cotangent-space-at-a-point-x}
Given a prestack $\sX$ and a point $(S \overset{x}{\ra} \sX) \in \Schaff_{/\sX}$ one says that $\sX$ admits a \emph{pro-cotangent space} at $x$ if the functor
\begin{align*}
    \Lift_{x}(\sX): \QCoh(S)^{\leq 0} & \ra \Spc \\
    \sF & \mapsto \Maps_{S/}(S_{\sF},\sX)
\end{align*}
preserves finite limits\footnote{Notice that it is enough to check that $\Lift_{x}(\sX)$ sends limits of the form $0 \underset{\sF_2}{\times}\sF_1$ in $\QCoh(S)^{\leq 0}$ to limits in $\Spc$.}. Here $S_{\sF}$ denotes the split square-zero extension associated to $\sF$ and
\[
    \Maps_{S/}(S_{\sF},\sX) := \ast \underset{\sX(S)}{\times}\sX(S_{\sF}),
\]
where $\ast \overset{x}{\ra} \sX(S)$.

We let $T^*_{x}\sX$ denote the object of $\Pro(\QCoh(S)^-)$ pro-corepresenting the functor $\Lift_{x}(\sX)$.
\end{defn}

\begin{parag}
\label{subsubsec:pro-cotangent-spaces}
We say that a prestack $\sX$ admits pro-cotangent spaces if it satisfies Definition \ref{defn:pro-cotangent-space-at-a-point-x} for every point $(S \overset{x}{\ra} \sX) \in \Schaff_{/\sX}$.
\end{parag}

\begin{parag}
The difference between pro-cotangent spaces and pro-cotangent complex is that the later requires that the objects $T^*_{x}\sX$ are compatible with base change.
\end{parag}

\begin{defn}
\label{defn:pro-cotangent-complex-at-a-point}
Given a prestack $\sX$ and a point $(S \overset{x}{\ra} \sX) \in \Schaff_{/\sX}$ one says that $\sX$ admits a \emph{pro-cotangent complex at $x$} if
\begin{condlist}
    \item $\sX$ admits a pro-cotangent space at $x$;
    \item for every map $f:S' \ra S$ in $\Schaff$ and every $\sF' \in \QCoh(S')^{\leq 0}$, the map
    \[
        \colim_{\sF \in {\rm QCoh}(S)^{\leq 0} \;|\; f^*(\sF) \ra \sF'} \Maps_{S/}(S_{\sF},\sX) \ra \Maps_{S'/}(S'_{\sF'},\sX)
    \]
    is an isomorphism of spaces.
\end{condlist}
\end{defn}

\begin{parag}
\label{subsubsec:pro-cotangent-complex}
We say that $\sX$ admits pro-cotangent complex if it admits a pro-cotangent complex at every point $(S \overset{x}{\ra} \sX) \in \Schaff_{/\sX}$. In this case one has an object $T^*\sX \in \Pro(\QCoh(\sX)^{-})$ called the \emph{pro-cotangent complex of $\sX$}.
\end{parag}

\begin{parag}
For a prestack $\sX$ that admits a pro-cotangent space at point $(S \overset{x}{\ra} \sX) \in \Schaff_{/\sX}$, we say that $\sX$ admits a \emph{cotangent space at $x$} if the functor $\Lift_{x}(\sX):\QCoh(S)^{\leq 0} \ra \Spc$ preserves cofiltered limits. This formally implies that
\[
    T^*_{x}\sX \in \QCoh(S)^-.
\]
\end{parag}

\begin{defn}
\label{defn:cotangent-complex-at-x}
Given a prestack $\sX$ and a point $(S \overset{x}{\ra} \sX) \in \Schaff_{/\sX}$ one says that $\sX$ admits a \emph{cotangent complex at $x$} if
\begin{condlist}
    \item $\sX$ admits a cotangent space at $x$;
    \item for every map $f:S' \ra S$ of affine schemes and every $\sF' \in \QCoh(S')^{\leq 0}$, the map
    \[
        \colim_{\sF \in {\rm QCoh}(S)^{\leq 0} \;|\; f^*(\sF) \ra \sF'} \Maps_{S/}(S_{\sF},\sX) \ra \Maps_{S'/}(S'_{\sF'},\sX)
    \]
    is an isomorphism of spaces.
\end{condlist}
\end{defn}

\begin{parag}
\label{subsubsec:cotangent-complex}
We will say that a prestack $\sX$ admits a \emph{cotangent complex} if it admits a cotangent complex at every point $(S \overset{x}{\ra} \sX) \in \Schaff_{/\sX}$. In this case one obtains an object $T^*\sX \in \QCoh(\sX)^-$ called the \emph{cotangent complex of $\sX$}.
\end{parag}

\begin{parag}
\label{subsubsec:convergent-and-laft-cotangent-complex}
The following result will be useful for us to determine if $\sX$ has a cotangent complex by checking conditions of Definition \ref{defn:cotangent-complex-at-x} in a smaller class of affine schemes and quasi-coherent sheaves on these schemes.
\end{parag}

\begin{lem}
\label{lem:cotangent-complex-from-ft-points}
Assume that $\sX$ is a convergent and locally almost of finite presentation prestack. Then $\sX$ admits a cotangent complex if and only if the following conditions hold:
\begin{condlist}
    \item for every $(S\overset{x}{\ra}\sX) \in (\Schaffconvft)_{/\sX}$ the functor
    \[
        \Lift_{x}(\sX):\Coh(S)^{\leq 0} \ra \Spc
    \]
    preserves all small limits;
    \item for every morphism $f:S' \ra S$ in $\Schaffconvft$ and every $\sF' \in \Coh(S')^{\leq 0}$, the map
    \[
        \colim_{\sF \in {\rm Coh}(S)^{\leq 0} \;|\; f^*(\sF) \ra \sF'} \Maps_{S/}(S_{\sF},\sX) \ra \Maps_{S'/}(S'_{\sF'},\sX)
    \]
    is an isomorphism of spaces.
\end{condlist}
We will say that $\sX$ admits a cotangent complex at a point $(S\overset{x}{\ra}\sX) \in (\Schaffconvft)_{/\sX}$ if conditions 1) and 2) hold.
\end{lem}

\begin{proof}
Notice that $\sX$ admits a cotangent space at a point $(S\overset{x}{\ra}\sX) \in \Schaff_{/\sX}$ if and only if the functor
\[
    \Lift_{x}(\sX):\QCoh(S)^{\leq 0 } \ra \Spc
\]
commutes with small limits. For any $\sF \in \QCoh(S)^{\leq 0}$ consider the diagram
\[
    \begin{tikzcd}
    \sX(S) \ar[r] & \lim_{n \geq 0}\sX(\tau^{\leq n}(S)) \\
    \sX(S_{\sF}) \ar[r] \ar[u] & \lim_{n \geq 0}\sX(\tau^{\leq n}(S)_{\tau^{\geq -n}\sF}) \ar[u] \\
    \Maps_{S/}(S_{\sF},\sX) \ar[r] & \lim_{n \geq 0}\Maps_{\tau^{\leq n}(S)/}(S_{\tau^{\geq -n}\sF},\sX)
    \end{tikzcd}
\]
where each column is the homotopy fiber over the point $x\in \sX(S)$. Since $\sX$ is convergent the two top rows are equivalences, hence so is the third and we have 
\[
    \Lift_{x}(\sX)(\sF) = \Maps_{S/}(S_{\sF},\sX) \simeq \lim_{n \geq 0}\Maps_{\tau^{\leq n}(S)/}(S_{\tau^{\geq -n}\sF},\sX).
\]

From this we obtain that $\Lift_{x}(\sX)$ preserves small limits in $\QCoh(S)^{\leq 0}$ if and only if each $\Lift_{x_n}(\sX)(-)=\Maps_{\tau^{\leq n}(S)/}(S_{(-)},\sX)$ preserve small limits in $\QCoh(S)^{\geq -n,\leq 0}$.

Thus, we can assume that $(S\overset{x}{\ra}\sX) \in \Schaffconv_{/\sX}$ is an eventually coconnective point, i.e.\ $S \in \Schaffn$ for some $n \geq 0$. Let $S \simeq \lim_I S_i$ be a presentation with $S_i \in \Schaffnfp$, given any object $\sF \in \QCoh(S)$ by \cite[Proposition 4.5.1.2.]{SAG} one has a compatible family $\{\sF_i \in \QCoh(S_i)\}_I$ recovering $\sF$ in the category $\QCoh(S) \simeq \lim_I\QCoh(S_i)$. More generally, given a diagram $\{\sF_j\}_J$ in $\QCoh(S)$ we let $\{\sF_{j,i}\in \QCoh(S_i)\}_{J\times I}$ be a system such that taking the limit at every $i \in I$ gives a diagram  $\{\lim_J\sF_{j,i} \in \QCoh(S_i)\}_I$ representing $\sF \in \QCoh(S)$. 

Consider the diagram
\[
    \begin{tikzcd}
    \sX(S) & \colim_{I^{\rm op}} \sX(S_i) \ar[l] \\
    \sX(S_{\sF}) \ar[u] & \colim_{(J\times I)^{\rm op}}\sX((S_{i})_{\sF_j}) \ar[l] \ar[u] \\
    \Maps_{S/}(S_{\sF},\sX) \ar[u] & \colim_{(J\times I)^{\rm op}}\Maps_{S_i/}((S_{i})_{\sF_j},\sX) \ar[u] \ar[l]
    \end{tikzcd}
\]
where each column is again a homotopy fiber over $x\in \sX(S)$. Notice that the diagram $J\times I$ is co-filtered and since $\sX$ is locally almost of finite presentation the two top horizontal morphism are equivalences, thus one has
\[
    \Lift_{x}(\sX)(\sF) = \Maps_{S/}(S_{\sF},\sX) \simeq \colim_{(J\times I)^{\rm op}}\Maps_{S_i/}((S_{i})_{\sF_j},\sX).
\]
From this equivalence it follows formally that $\Lift_{x}(\sX)(-)$ preserves small limits if and only if each $\Maps_{S_i/}((S_{i})_{-},\sX)$ preserves small limits.

The argument condition 2) is similar and we leave it to the reader.
\end{proof}

\begin{defn}
\label{defn:def-theory}
One says that a prestack $\sX$ has \emph{deformation theory} if
\begin{condlist}
    \item $\sX$ is convergent;
    \item $\sX$ admits a pro-cotangent complex;
    \item $\sX$ is infinitesimally cohesive.
\end{condlist}
One says that $\sX$ has \emph{corepresentable deformation theory} if $\sX$ admits a cotangent complex.
\end{defn}

For our convenience we collect some of the relations between properties of a prestack and that of its cotangent complex.

\begin{prop}
\label{prop:properties-cotangent-complex}
\begin{resultlist}
    \item Assume that $f:\sX \ra \sY$ is locally almost of finite presentation and admits a cotangent complex, then $T^*(\sX/\sY) \in \QCoh(\sX)$ is almost perfect;
    \item Assume that $f:\sX \ra \sY$ admits deformation theory and that $T^{*}(\sX/\sY) \in \QCoh(\sX)$ is almost perfect, then $f$ is locally almost of finite presentation.
\end{resultlist}
\end{prop}

\begin{proof}
(a) and (b) are Corollary 17.4.2.2 (1) and (2) of \cite{SAG}, respectively.
\end{proof}

\begin{lem}
\label{lem:coherent-cohomology-implies-coherent}
Let $\sX$ be a prestack. Assume that 
\begin{condlist}
    \item $\sX$ admits deformation theory;
    \item there exists $n \geq 0$ such that $\sX$ admits $(-n)$-connective tangent spaces\footnote{See \cite[Chapter 1, \S 3.1]{GRII} for a definition. Alternatively, this is the same as requiring that in condition $(\beta)$ below $H^k(T^*_{x_0}(\sX))$ for every $k > n$.}, and
    \item[($\beta$)] for every $x:T\ra \classical{\sX}$ where $T$ is a classical affine scheme of finite type one has
    \[
        H^k(T^*_{x}(\sX)) \in \Coh(T)^{\heartsuit}
    \]
    for all $k \in \bZ$.
\end{condlist}
Then for every $x:T\ra \classical{\sX}$ where $T$ is a classical affine scheme of finite type one has
\[
    T^*_{x}(\sX) \in \Coh(T).
\]
\end{lem}

\begin{proof}
Notice that by \cite[Chapter 1, Corollary 3.2.4]{GRII} it is enough to check that the functor
\begin{align*}
    \QCoh(T)^{\leq 0} & \ra \Spc \\
    \sF & \mapsto \Maps_{T/}((T)_{\sF},\sX)
\end{align*}
commutes with cofiltered limits. Since the functor $\sF \rightsquigarrow S_{\sF}$ sends limits that exists in $\QCoh(T)^{\leq 0}$ to colimits, it is enough to check that for every cofiltered diagram $\{\sF_{i}\}_I$ in $\QCoh(S)^{\leq 0}$ whose limit is $\sF$ the canonical map
\[
    \sX(T_{\sF}) \ra \lim_{I}\sX(T_{\sF_i})
\]
is an equivalence. Because $\sX$ is convergent it is enough to check that for every $m \geq 0$ one has an equivalence
\begin{equation}
    \label{eq:m-coconnective-part-of-sX-commutes-with-filtered-colimits}
    \sX(\tau^{\leq m}(T_{\sF})) \ra \lim_{I}\sX(\tau^{\leq m}(T_{\sF_i})).
\end{equation}
Notice that for any $\sG \in \QCoh(T)^{\leq 0}$ one has an equivalence
\[
    \tau^{\leq m}(T_{\sG}) \simeq \tau^{\leq m}(T)_{\tau^{\geq -m}(\sG)}.
\]
Thus, in (\ref{eq:m-coconnective-part-of-sX-commutes-with-filtered-colimits}) we can assume that $\sF \in \QCoh(T)^{\geq -m,\leq 0}$. Moreover, since $\QCoh(T)^{\geq -m}$ is stable under limits we can assume that each $\sF_{i} \in \QCoh(T)^{\geq -m,\leq 0}$. Thus, we have that
\[
      \Maps_{T/}(T_{\sF},\sX) \simeq \Hom_{\Pro(\Coh(T))}(T^*_{x}\sX,\sF) \;\;\; \mbox{and} \;\;\; \lim_{I}\Maps_{T/}(T_{\sF_i},\sX) \simeq \lim_I\Hom_{\Pro(\Coh(T))}(T^*_{x}\sX,\sF_i).
\]
The coconnectivity assmption on $\sF$ and $\{\sF_{i}\}_I$ reduces the above to
\[
    \Hom_{\Pro(\Coh(T))}(T^*_{x}\sX,\sF) \simeq \Hom_{\Pro(\Coh(T))}(\tau^{\geq -m}(T^*_{x}\sX),\sF) 
\]
and
\[
\lim_I\Hom_{\Pro(\Coh(T))}(T^*_{x}\sX,\sF_i) \simeq \lim_I\Hom_{\Pro(\Coh(T))}(\tau^{\geq -m}(T^*_{x}\sX),\sF_i).
\]
However, we notice that if $H^k(T^*_{x}\sX)$ is coherent for every $k\in \bZ$ then
\[
    \tau^{\geq -m}(T^*_{x}\sX) \in \Coh(T)^{\geq -m,\leq n},
\]
in other words, given a presentation $T^*_{x}(\sX) \simeq \lim_{J}\sG_j$ where $\sG_{j} \in \Coh(T)^{\leq n}$, then $\tau^{\geq -m}(T^*_{x}\sX) \simeq \tau^{\geq -m}(\sG_{j_m})$ for some $j_m \in \bZ$. Thus, one obtains
\begin{align*}
    \Hom_{\Pro(\Coh(T))}(T^*_{x}\sX,\sF) & \overset{\simeq}{\ra} \Hom_{\Pro(\Coh(T))}(\tau^{\geq -m}(T^*_{x}\sX),\sF) \\
    & \simeq \Hom_{\Coh(T)}(\tau^{\geq -m}(\sG_{j_m}),\sF) \\
    & \simeq \lim_I\Hom_{\Coh(T))}(\tau^{\geq -m}(\sG_{j_m}),\sF_i) \\
    & \simeq \lim_I\colim_{J^{\rm op}}\Hom_{\Coh(T))}(\tau^{\geq -m}(\sG_{j}),\sF_i) \\
    & \simeq \lim_I\Hom_{\Pro(\Coh(T))}(\tau^{\geq -m}(T^*_{x}\sX),\sF_i) \\
    & \simeq \lim_I\Hom_{\Pro(\Coh(T))}(T^*_{x}\sX,\sF_i)
\end{align*}
This finishes the proof.
\end{proof}

\begin{warning}
In general, given an object $\sG \in \Pro(\Coh(S)^-)$ such that for every $k \in \bZ$ one has
\[
    H^k(\sG) \in \Coh(S)^{\heartsuit}
\]
does \emph{not} imply that $\sG \in \Coh(S)^-$. Indeed, simply consider $\sG := \lim_{n \geq 0}\sO_{S}\oplus \sO_{S}[1] \oplus \cdots \oplus \sO_{S}[n]$.
\end{warning}

\subsection{Relation between deformation theory and homogeneity}

There is a big overlap between some notions of deformation theory and homogeneity. We start with a result connecting formal cohesiveness to homogeneity with respect to local Artinian schemes.

\begin{prop}
\label{prop:equivalent-formally-cohesive}
Let $\sX$ be a prestack which admits a cotangent complex at every local Artinian affine scheme of finite type. Then the following are equivalent:
\begin{equivlist}
    \item $\sX$ is $W^{\rm lA}_{\rm all}$-homogeneous;
    \item $\sX$ is formally cohesive;
    \item for every point $(S \overset{x}{\ra} \sX) \in \SchafflA_{/\sX}$ and $(T^*(S) \ra \sF) \in \QCoh(S)^{\leq -1}_{T^*(S)}$ one has an isomorphism
    \[
        \Maps_{S/}(\SqZ(T^*(S) \ra \sF),\sX) \simeq \left\{\mbox{null homotopies of }T^*_{x}\sX \overset{(dx)^*}{\ra} T^*(S) \ra \sF\right\}
    \]
    where $\SqZ(T^*(S) \ra \sF)$ is the square-zero extension of $S$ determined by $\sF$\footnote{See \cite[Chapter 1, \S 5.1]{GRII} for a definition of square-zero extension.}.
\end{equivlist}
\end{prop}

\begin{proof}
The implication (1) $\Rightarrow$ (2) is tautological once we notice that any closed embedding of local Artinian schemes is automatically nilpotent.

For (2) $\Rightarrow$ (1) we first notice that in the diagram (\ref{eq:P-nil-square}) $\jmath_1$ being a nilpotent embedding implies that $T'_1$ and $T'_2$ are also local Artinian schemes. By definition $T_1,T_2,T'_1$ and $T'_2$ are all eventually coconnective affine schemes. Thus by \cite[Chapter 1, Proposition 5.5.3]{GRII} we can reduce to the case where $\jmath_{1}$ is a square-zero extension, i.e.\ there exists $\gamma_1: T^*(T_1) \ra \sF_1$ for some $\sF_1 \in \QCoh(T_1)^{\leq -1}$ such that $T'_1 \simeq T_1 \sqcup_{(T_1)_{\sF_1}}T_1$ where the maps $\begin{tikzcd} (T_1)_{\sF_1} \ar[r,yshift=1.5pt] \ar[r,yshift=-1.5pt] & T_1 \end{tikzcd}$ are the canonical projection and the composite $(T_1)_{\sF_1} \overset{\gamma_1}{\ra} (T_1)_{T^*(T_1)} \overset{\mathfrak{d}}{\ra} T_1$ (see \cite[Chapter 1, \S 4.5]{GRII} for the definition of $\mathfrak{d}$). Since both morphisms $\begin{tikzcd} (T_1)_{\sF_1} \ar[r,yshift=1.5pt] \ar[r,yshift=-1.5pt] & T_1 \end{tikzcd}$ are closed embeddings\footnote{Notice however that the morphism $(T_1)_{T^*(T_1)} \overset{\mathfrak{d}}{\ra} T_1$ is not a closed embedding, but its precomposition with $(T_1)_{\sF_1} \overset{\gamma_1}{\ra} (T_1)_{T^*(T_1)}$ is.} by hypothesis we have an isomorphism
\[
    \sX(T'_1) \overset{\simeq}{\ra} \sX(T_1)\underset{\sX((T_1)_{\sF_1})}{\times}\sX(T_1).
\]
The pushout of a square-zero extension via any morphism $f:T_1 \ra T_2$ gives a square-zero extension, i.e.\ $T'_2 \simeq T_2 \sqcup_{(T_2)_{f_*\sF_1}}T_2$, So one has
\[
    \sX(T'_2) \overset{\simeq}{\ra} \sX(T_2)\underset{\sX((T_2)_{f_*\sF_1})}{\times}\sX(T_2).
\]
Recall that we need to check that
\[
    \Maps_{T_2/}(T'_2,\sX) \ra \Maps_{T_1/}(T'_1,\sX)
\]
is an isomorphism for each point $y_2:T_2 \ra \sX$.
Now we notice that since we assumed that $\sX$ admits a pro-cotangent space at every point, one has 
\[
    \sX(T'_2) \underset{\sX(T_2)}{\times}\{y_2\} \simeq \Map(T^*_{y_2}\sX,f_*\sF_1[-1]) \;\;\; \mbox{and} \;\;\; \sX(T'_1)\underset{\sX(T_1)}{\times}\{y_2\} \simeq \Map(T^*_{f\circ y_2}\sX,\sF_1[-1]).
\]
Finally, the assumption that $\sX$ actually admits a pro-cotangent complex and the adjunction $(f^*,f_*)$ gives that the two mapping spaces above are isomorphic.

For the proof of the equivalence (2) $\Longleftrightarrow$ (3) the same argument as that of \cite[Proposition 17.3.6.1]{SAG} works, where we just point out that in our situation we don't need to assume that $\sX$ is convergent. We leave the details to the reader.
\end{proof}

The following is the somewhat global version of the result above:

\begin{prop}
\label{prop:affine-homogeneous-is-deformation-theory}
Let $\sX$ be a convergent prestack, then the following are equivalent:
\begin{equivlist}
    \item $\sX$ is $W_{\rm all}$-homogeneous;
    \item $\sX$ admits deformation theory, i.e.\ $\sX$ admits a pro-cotangent complex and is infinitesimally cohesive.
\end{equivlist}
\end{prop}

\begin{proof}
The direction (1) $\Rightarrow$ (2) is \cite[Chapter 1, Proposition 7.2.5]{GRII}. 

For (2) $\Rightarrow$ (1) we notice that since $\sX$ is convergent it is enough to consider $T_1$ and $T_2$ eventually coconnective in the diagram (\ref{eq:P-nil-square}), which implies that $T'_1$ and $T'_2$ are also eventually coconnective. Since any nilpotent embedding between two eventually coconnective schemes can be obtained as a finite succession of square-zero extension, by \cite[Chapter 1, Corollary 6.3.5]{GRII} we are done.
\end{proof}

\subsection{Smoothness}

In proving the representability result one is faced with the following conundrum. One wants to construct a smooth morphism \emph{into our prestack} from a disjoint union of affine schemes. The problem is that one only knows how to define smoothness for a morphism (cf.\ \cite[Chapter 2, \S 4.1.7]{GRI}) that we already know is representable\footnote{Namely a morphism $f:\sX \ra S$ from an $n$-geometric stack to an affine scheme is smooth if given an atlas $\sZ \ra \sX$ the composite $\sZ \ra S$, which is an $(n-1)$-geometric stack is smooth, then the definition goes by induction.} which is not the case in our situation, since we do not have any representability assumption on the diagonal.

The strategy we take in this section is to start with the notion of formal smoothness, which make sense for arbitrary prestacks, then impose a further finiteness condition on the morphism to obtain the notion of differentially smooth. When the prestack admits deformation theory, one can describe formal smoothness and the finiteness condition, whence differentially smooth, in terms of the cotangent complex. Finally, we perform a sanity check that in the case when the prestack is geometric differentially smooth is the same as smooth in the sense from \cite[Chapter 1, \S 4.1]{GRI}.

\begin{rem}
There is another discussion to be had in the case where we don't assume our base ring to be a $\bQ$-algebra, since in that set up one also has the notion of fiber smooth, which in general does not agree with differentially smooth. In that context one can also prove a version of the \hyperref[thm:main-result]{Main Theorem}, however one needs to restart the discussion by reconsidering the definition of a geometric stack, since we can now start the inductive process by considering either the notion of fiber smoothness or differentially smoothnes, which don't agree in general. We will leave the investigation of this question for some future work.
\end{rem}

\begin{parag}
Recall that a morphism $f:\sX \ra \sY$ between prestacks is said to be \emph{formally smooth} if given any nilpotent embedding of affine schemes $S \ra S'$ the canonical map
\[
    \sX(S') \ra \sX(S) \underset{\sY(S)}{\times}\sY(S')
\]
induces a surjection on connected components.
\end{parag}

Before discussing the notion for arbitrary prestacks we recall the following characterization of smoothness for morphisms between affine schemes.

\begin{lem}
\label{lem:smoothness-for-affines}
Consider a morphism $f:U \ra S$ between affine schemes, then the following are equivalent:
\begin{equivlist}
    \item $f$ is smooth, i.e.\ $f$ is flat and $\classical{f}$ is smooth;
    \item $f$ is fiber smooth, i.e.\ $f$ is flat, almost of finite presentation and for every field point $u:\Spec\,k \ra U$ one has $H^{-1}(T^*_u(T/S)) = 0$;
    \item $f$ is differentially smooth, i.e.\ $f$ is almost of finite presentation and formally smooth. 
\end{equivlist}
\end{lem}

\begin{proof}
The equivalence (ii) $\Longleftrightarrow$ (iii) is \cite[Proposition 11.2.4.4]{SAG}, notice that here we are using the standing assumption that or base has characteristic $0$.

For (i) $\Longleftrightarrow$ (ii) we notice that since in both cases we assume that $f$ is flat one has
\[
    \classical{f}:\classical{U} \overset{\simeq}{\ra} U\underset{S}{\times}\classical{S} \ra \classical{S}
\]
By Remark 11.2.3.5 \cite{SAG} one has that $f$ is fiber smooth if and only if $\classical{f}$ is fiber smooth, when seen as a morphism of (derived) affine schemes. Then Proposition 11.2.4.1 gives that $\classical{f}$ is fiber smooth if and only $\classical{f}$ satisfies one of the more widely know definitions of smoothness for a morphism of commutative rings, e.g.\ \cite[Tag 00T1]{stacks-project} that says that $\classical{f}$ is smooth if the na\"ive cotangent complex of $\classical{f}$, i.e.\ $\tau^{\geq -1}(T^{*}(\classical{U}/\classical{S}))$ is quasi-isomorphic to a finite projective module in degree 0, that is exactly statement (3) in \cite[Proposition 11.2.4.1]{SAG}.
\end{proof}

\begin{defn}
\label{defn:differentially-smooth}
Given a morphism $f:\sX \ra \sY$ of prestacks we say that $f$ is \emph{differentially smooth} if
\begin{defnlist}
    \item $f$ is formally smooth;
    \item $f$ is locally almost of finite presentation.
\end{defnlist}
\end{defn}

\begin{lem}
Assume that $f$ admits deformation theory. Then one has
\begin{resultlist}
    \item $f$ is formally smooth if and only if for any $(S \overset{x}{\ra} \sX) \in \Schaff_{/\sX}$ and $\sF \in \QCoh(S)^{\heartsuit}$ one has
    \[
    Maps_{\QCoh(S)^-}(T^*_{x}(\sX/\sY),\sF) \in \Vect^{\leq 0};
    \]
    \item $f$ is locally almost of finite presentation if and only if $T^*(\sX/\sY)$ is an almost perfect object of $\QCoh(\sX)$.
\end{resultlist}
\end{lem}

\begin{proof}
(1) is Proposition 7.3.3 in \cite{GRII}.

(2) is Corollary 17.4.2.2 in \cite{SAG}.
\end{proof}

\begin{rem}
We notice that often in \cite{DAG} one has the assumption that $T^*_x(\sX/\sY)$ is perfect and the dual of a connective object in $\QCoh(S)$, we notice that this gives
\[
    Maps_{\QCoh(S)^-}(T^*_{x}(\sX/\sY),\sF) \simeq \Gamma(S,T^*(\sX/\sY)^{\vee}\otimes \sF) \in \Vect^{\leq 0},
\]
so in particular it implies that $\sX \ra \sY$ is formally smooth at the point $x:S \ra \sX$, i.e.\ the composite $S \overset{x}{\ra} \sX \ra \sY$ is formally smooth.
\end{rem}

\begin{prop}
\label{prop:differentially-smooth-is-smooth}
Assume the morphism $f:\sX \ra \sY$ is geometric. Then the following are equivalent:
\begin{equivlist}
    \item $f$ is differentially smooth;
    \item $f$ is smooth.
\end{equivlist}
\end{prop}

\begin{proof}
Let $f:\sX \ra \sY$ be an $n$-geometric morphism. By definition $f$ is smooth if for every affine scheme point $T \ra \sY$ the pullback morphism $\sX\underset{\sY}{\times}T \ra T$ is smooth, that is it is $k$-geometric for some $k\geq 0$ and smooth. Recall that an $n$-geometric morphism to an affine scheme is smooth if there exists an atlas $\sU \simeq \sqcup_I U_i \ra \sX$ with $U_i$ affine and $U_i\underset{\sY}{\times}T \ra \sX\underset{\sY}{\times}T \ra T$ smooth for each $i \in I$. Since $f$ is $n$-geometric one has that each $U_i\underset{\sY}{\times}T \ra T$ is $(n-1)$-geometric. Thus, by inducing on the statement: ($k$-geometric and weakly smooth if and only if $k$-geometric and smooth) we are reduced to the case of $n=-1$.

The statement for $n=-1$ reduces to Lemma \ref{lem:smoothness-for-affines}.
\end{proof}


\begin{parag}
Finally, we need a little result about lifting of Henselian pairs against smooth morphisms. Recall that $(R,I)$ where $R$ is a (derived) ring and $I \subset R$ is a Henselian pair if every diagram 
\[
    \begin{tikzcd}
    \Spec(R/I) \ar[r] \ar[d] & \Spec(B) \ar[d,"f"] \\
    \Spec(R) \ar[r] \ar[ru,dashed] & \Spec(A)
    \end{tikzcd}
\]
where $f$ is \'etale, admits a dashed arrow filling the diagram.
\end{parag}

\begin{lem}
\label{lem:Henselian-lift-against-smooth-geometric-stacks}
Let $f:\sX \ra \sY$ be a smooth geometric morphism between prestacks, given any diagram
\[
    \begin{tikzcd}
    \Spec(R/I) \ar[r] \ar[d] & \sX \ar[d,"f"] \\
    \Spec(R) \ar[r] & \sY
    \end{tikzcd}
\]
where $(R,I)$ is a Henselian pair, there exists an \'etale cover $\Spec(R') \ra \Spec(R)$ such that the following diagram
\[
    \begin{tikzcd}
    \Spec(R/I)\underset{\Spec(R)}{\times}\Spec(R') \ar[r,"x"] \ar[d] & \sX \ar[d,"f"] \\
    \Spec(R') \ar[r] \ar[ru,dashed] & \sY
    \end{tikzcd}
\]
admits a lifting.
\end{lem}

\begin{proof}
The question is local on $\sY$, so we can assume that $\sY \simeq S$ is affine. Since $f:\sX \ra S$ is $n$-geometric, let $\sU \ra \sX$ denote a smooth atlas where $\sU$ is $(n-1)$-geometric stacks, so by induction it is enough to consider the case $n=0$. 

Thus, we can take $\sU \ra \sX$ an \'etale atlas where $\sU:=\sqcup_I U_i$ is a disjoint union of affine schemes, by definition there exists an \'etale cover $\Spec\,\bar{R}' \ra \Spec(R/I)$ such that one has a factorization 
\[
x':\Spec\,\bar{R}' \ra U_i \ra \sX.
\]

Since $(R,I)$ is a Henselian pair, there exists an \'etale cover $\Spec\,R' \ra \Spec\,R$ such that $\Spec\,\bar{R}' \simeq \Spec(R/I) \underset{\Spec\,R}{\times}\Spec\,R'$, thus it is enough to find a lift for the following diagram
\[
    \begin{tikzcd}
    \Spec\,\bar{R}' \ar[r,"x'"] \ar[d] & U_i \ar[d] \\
    \Spec\,R' \ar[r] \ar[ru,dashed] & S
    \end{tikzcd}
\]
where $U_i \ra S$ is an affine morphism. But that is Elkik's Theorem \cite[\S II]{Elkik}.

\end{proof}

\section{Bootstrapping}

\subsection{Descent and homogeneity}

The following is an analogue of Lemma 1.9 in \cite{Hall-Rydh-algebraicity}.

\begin{lem}
\label{lem:descent-implies-formally-cohesive}
Assume that $\sX$ is $W^{\rm lA}_{\rm triv.}$-homogeneous and that $\sX$ is a sheaf for the \'etale topology, then $\sX$ is $W^{\rm lA}_{\rm all}$-homogeneous.
\end{lem}

\begin{proof}
We need to check that $\sX$ sends any pushout diagram
\[
    \begin{tikzcd}
    T_1 \ar[r,"f"] \ar[d,hook] & T_2 \ar[d,hook] \\
    T'_1 \ar[r] & T'_2
    \end{tikzcd}
\]
where $f$ is a local morphism to a pullback diagram of spaces. Let $t_i \in T_i$ and $t'_i \in T'_i$, for $i=1,2$, denote the closed points and $\kappa(t_i)$ and $\kappa(t'_i)$ their residue fields. We notice that there exists a finite extension $\Spec\,\widetilde{\kappa(t'_2)} \ra \Spec\,\kappa(t'_2)$ such that 
\[
\Spec\,\kappa(t_1)\underset{\Spec\,\kappa(t'_2)}{\times}\Spec\,\widetilde{\kappa(t'_2)} \overset{\simeq}{\ra} \Spec\,\kappa(t_1).
\]
Since we assumed that our base has characteristic $0$ any finite extension is separable, hence by \cite[Corollaire 10.3.2]{EGAIII1} there exists an \'etale cover $\bar{T'_2} \ra \classical{T'_2}$ such that $\bar{T'_2}\underset{T'_2}{\times}\Spec\,\kappa(t'_2) \simeq \Spec\,\widetilde{\kappa(t'_2)}$. Furthermore, by \cite[Theorem 7.5.0.6]{HA} there exists an \'etale cover $\widetilde{T'_2} \ra T_2$ such that
\[
    \widetilde{T'_2}\underset{T'_2}{\times}\classical{T'_2} \overset{\simeq}{\ra} \bar{T'_2}.
\]
Thus, we obtain that considering the pullbacks $\widetilde{T_1} := T_1 \underset{T'_2}{\times}\widetilde{T'_2}$, $\widetilde{T_2} := T_2 \underset{T'_2}{\times}\widetilde{T'_2}$ and $\widetilde{T'_1} := T'_1 \underset{T'_2}{\times}\widetilde{T'_2}$ by the assumption that $\sX$ is $W^{\rm lA}_{\rm triv.}$-homogeneous one has an equivalence
\[
    \sX(\widetilde{T'_2}) \overset{\simeq}{\ra} \sX(\widetilde{T'_1})\underset{\sX(\widetilde{T_1})}{\times}\sX(\widetilde{T_2}).
\]
Thus, by \'etale descent of $\sX$ we are done.
\end{proof}

\subsection{Local to global homogeneity}

We start with the following analogue of Lemma B.2 from \cite{Hall-Rydh-algebraicity}.

\begin{lem}
\label{lem:H1-for-laft-affines}
    Let $f:\sX \ra \sY$ be a morphism of $S$-prestacks, the following are equivalent:
    \begin{equivlist}
        \item $f$ satisfies $(H^{\rm P}_1)$;
        \item the diagonal morphism $\Delta_{\sX/\sY}:\sX \ra \sX\underset{\sY}{\times}\sX$ is $P$-homogeneous.
    \end{equivlist}
    In addition, if $\sX,\sY$ are Zariski stacks and locally almost of finite presentation, and $P$ is Zariski local, then we have 
    \begin{equivlist}
        \item[(3)] we only need to check condition (i) or (ii) for affine schemes which are almost of finite presentation over $S$.
    \end{equivlist}
    In particular, given a prestack $\sX \ra S$ if $\Delta_{\sX/S}$ is geometric then $\sX$ satisfies $(H^{\rm W_{\rm all}}_{1})$.
\end{lem}

\begin{proof}
(1) $\Rightarrow$ (2). By Lemma \ref{lem:compatible-relative-absolute} it is enough to prove that for every affine scheme $U \ra \sX\underset{\sY}{\times}\sX$ the base change $D_{f,U}:= \sX \underset{(\sX\underset{\sY}{\times}\sX)}{\times}U$ is $P$-homogeneous. So we need to prove that for any pushout diagram as in (\ref{eq:P-nil-square}) the canonical morphism
\[
    D_{f,u}(T'_2) \ra D_{f,u}(T'_1)\underset{D_{f,U}(T_1)}{\times}D_{f,U}(T_2)
\]
induces a surjection on connected components. Assume to the contrary that there are $x,y \in \pi_0(D_{f,u}(T'_2))$ whose images $\overline{x_1}$ and $\overline{x_2}$ on the right-hand side of the above map agree. This means that they belong to the same connected component of
\begin{equation}
    \label{eq:D-pullback}
    \left(\sX \underset{(\sX\underset{\sY}{\times}\sX)}{\times}U\right)(T'_1) \underset{\left(\sX \underset{(\sX\underset{\sY}{\times}\sX)}{\times}U\right)(T_1)}{\times}\left(\sX \underset{(\sX\underset{\sY}{\times}\sX)}{\times}U\right)(T_2)
\end{equation}
By the same argument as in the proof of Lemma \ref{lem:compatible-relative-absolute} one can rewrite (\ref{eq:D-pullback}) as
\begin{equation}
    \label{eq:glued-diagonal}
    \left(\sX(T'_1)\underset{\sX(T_1)}{\times}\sX(T_2)\right)\underset{(\sX\underset{\sY}{\times}\sX)(T'_1)\underset{(\sX\underset{\sY}{\times}\sX)(T_1)}{\times}(\sX\underset{\sY}{\times}\sX)(T_2)}{\times}\left(U(T'_1)\underset{U(T_1)}{\times}U(T_2)\right).
\end{equation}
Since $\sX$ satisfies $(H^{1}_P)$ and $U$ is $P$-homogeneous, one obtains that the images $\overline{x_1}$ and $\overline{x_2}$ in (\ref{eq:glued-diagonal}) agree on the terms $\left(\sX(T'_1)\underset{\sX(T_1)}{\times}\sX(T_2)\right)$ and $\left(U(T'_1)\underset{U(T_1)}{\times}U(T_2)\right)$. To finish the argument, we notice that to check that two elements of $\sX(T'_2)$ agree on $\sY(T'_2)$ it is enough to check that they agree on $\sY(T'_1)\underset{\sY(T_1)}{\times}\sY(T_2)$ from the pullback diagram
\[
    \begin{tikzcd}
    \sX(T'_2)\underset{\sY(T'2)}{\times}\sX(T'_2) \ar[r] \ar[d] & \sX(T'_2)\underset{\sY(T'_1)\underset{\sY(T_1)}{\times}\sY(T_2)}{\times}\sX(T'_2) \ar[d] \\
    \sY(T'_2) \ar[r] & \sY(T'_2)\underset{\sY(T'_1)\underset{\sY(T_1)}{\times}\sY(T_2)}{\times}\sY(T'_2)
    \end{tikzcd}
\]
and the fact that we started with elements of $D_{f,U}(T'_2)$, i.e.\ morphisms $x'_1,x'_2:T'_2 \ra U$ and $x''_1,x''_2:T'_2 \ra \sX$ which agree on $\sX\underset{\sY}{\times}\sX$.

To check that $D_{f,U}$ satisfies $(H^{\rm P}_2)$ we notice that we need to fill the dashed arrow in the diagram 
\[
    \begin{tikzcd}
    T_1 \ar[r] \ar[d,hook] & T_2 \ar[d,hook] \ar[rr] & & D_{f,U} \ar[r] \ar[d] & \sX \ar[d] \\
    T'_1 \ar[r] & T'_2 \ar[rru,dashed] \ar[rrrd] \ar[rrrd,xshift=-3pt,yshift=-5pt] \ar[rr,"g"] & & U \ar[r] & \sX\underset{\sY}{\times}\sX \ar[d,xshift=-3pt,"p_1"'] \ar[d,xshift=3pt,"p_2"] \\
    & & & & \sX
    \end{tikzcd}
\]
By definition of pullback it is enough to construct a dashed arrow into $\sX$. However that follows from noticing that the two bottom composites agree since $\sX$ satisfies $(H^1_P)$, so the morphism $g$ lifts to $\sX$ and we are done.

(2) $\Rightarrow$ (1) consider two morphisms $x_1,x_2:T'_2 \ra \sX$ whose restriction give the same commutative diagram
\[
    \begin{tikzcd}
    T_1 \ar[r] \ar[d,hook] & T_2 \ar[rdd] \ar[d,hook] & \\
    T'_1 \ar[rrd] \ar[r] & T'_2 \ar[rd,dashed] & \\
    & & \sX
    \end{tikzcd}
\]
We need to prove that $x_1$ and $x_2$ are nullhomotopic, i.e.\ that the dashed arrow above can be completed by either $x_1$ or $x_2$ into a commutative diagram. Consider the pullback $D_{f,T'_2}:=\sX \underset{\sX\underset{\sY}{\times}\sX}{\times}T'_2$, then for $i=1,2$ there exists a lift
\[
    \begin{tikzcd}
    T_1 \ar[r] \ar[d,hook] & T_2 \ar[d,hook] & D_{f,T'_2} \ar[d] \ar[r] & T'_2 \ar[d]\\
    T'_1 \ar[r] & T'_2 \ar[ru,dashed] \ar[r,"x_i"] & \sX \ar[r] & \sX\underset{\sY}{\times}\sX
    \end{tikzcd}
\]
since $D_{f,T'_2}$ is $P$-homogeneous. Moreover, since the lift is uniquely determined by the restriction of the diagram to $T'_1 \hookleftarrow T_1 \ra T_2$, $x_1$ and $x_2$ are then both given by the composite $T'_2 \ra D_{f,T'_2} \ra \sX$, which gives the result.

Assertion (3) is clear because we can reduce to considering diagrams (\ref{eq:P-nil-square}) where all affine scheme are almost of finite presentation over $S$.

The last assertion follows from Proposition \ref{prop:n-geometric-has-deformation-theory} and Proposition \ref{prop:affine-homogeneous-is-deformation-theory}, which gives (2) hence (1).
\end{proof}

The following is an analogue of Lemma B.3 in \cite{Hall-Rydh-algebraicity}.

\begin{lem} 
\label{lem:H2-for-Henselian-schemes}
    Consider a locally almost of finite presentation prestack $\sX$ such that $\sX$ satisfies $(H^{\rm P}_1)$, then the following are equivalent
    \begin{equivlist}
        \item $\sX$ satisfies $(H^{\rm P}_2)$;
        \item $\sX$ satisfies condition $(H^{\rm P}_2)$ for affine schemes in the subcategory $\Schaffft$;
        \item $\sX$ satisfies $(H^{\rm P}_2)$ for diagrams
        \begin{equation}
            \label{eq:Henselian-pushout-square}
            \begin{tikzcd}
                T_1 \ar[r,"f"] \ar[d,"\jmath"'] & T_2 \ar[d] \\
                T'_1 \ar[r] & T'_2
            \end{tikzcd}
        \end{equation}
        where $T_2$ is the Henselisation of an affine scheme of finite type at a closed point, and $f$ and $\jmath$ are finitely presented.
    \end{equivlist}
\end{lem}

\begin{proof}
The implications (1) $\Rightarrow$ (2) $\Rightarrow$ (3) are tautological. 

    For (3) $\Rightarrow$ (2), consider a pushout diagram
    \[
        \begin{tikzcd}
            S_1 \ar[r,"f"] \ar[d,hook,"\jmath_1"'] & S_2 \ar[d,"\jmath_2"] \\
            S'_1 \ar[r] & S'_2
        \end{tikzcd}    
    \]
    where $\jmath_1$ is a nilpotent embedding and $f$ belongs to the class $P$. Consider a closed point $s_2 \in S_2$ and let $T_2:=(S_2)^{\rm H.}_{s_2}$ be the Henselisation of $S_2$ at the point $s_2$, i.e.\ 
    \[
        T_2 \simeq \lim_{(U_2,u_2 \in |U_2|) \;| \; U_2\mbox{ an \'etale nbd. of }s_2}U_2,    
    \]
    and the induced morphism of residue fields $\Spec \kappa (u_2) \ra \Spec \kappa(s_2)$ is an isomorphism. We first notice that the Henselisation of $S_2$ lifts uniquely to a Henselisation $T'_2$ of $S'_2$. Indeed, by definition the morphism $\bar{f}: T_2 \ra S_2$ is pro-\'etale, i.e.\ the (co-)filtered limit of \'etale morphism, by Proposition 2.3.3 \cite{BS-Pro-etale} the relative cotangent complex $T^*(T_2/S_2)$ vanishes. By \cite[Proposition 5.5.3]{GRII} any nilpotent embedding $S_2 \hra S'_2$ is an iteration of square-zero extensions and since $\jmath_2$ is finitely presented there are only finitely many of them. In particular, we can assume that $S_2\hra S'_2$ is a square-zero extension, hence it is determined by a morphism 
    \[
        \alpha: T^*(S_2) \ra \sF    
    \]
    for some $\sF \in \QCoh(S_2)^{\leq -1}$. The vanishing of $T^*(T_2/S_2)$ implies that one has an equivalence $p^*_2(T^*(S_2)) \simeq T^*(T_2)$ which gives a morphism 
    \[
        (\bar{f})^*(\alpha): T^*(T_2) \ra (\bar{f})^*(\sF),
    \]
    that is a square-zero extension of $T_2$. Moreover, by the functoriality of square-zero extension we know that the following diagram commutes
    \[
        \begin{tikzcd}
            S_1 \ar[r] \ar[d,"\jmath"'] & T_2 \ar[d] \ar[r,"\bar{f}"] & S_2 \ar[d] \\
            S'_1 \ar[r] & T'_2 \ar[r] & S'_2
        \end{tikzcd}
    \]

    Consider the diagram
    \[
        \begin{tikzcd}
            \sX(S'_2) \ar[r] \ar[d] & \sX(T'_2) \ar[d] \\
            \sX(S'_1)\underset{\sX(S_1)}{\times}\sX(S_2) \ar[r,"\gamma"] & \sX(S'_1)\underset{\sX(S_1)}{\times}\sX(T_2).
        \end{tikzcd}    
    \]
    Given an element $y \in \pi_0(\sX(S'_1) \underset{\sX(S_1)}{\times} \sX(S_2))$, let $\gamma(y) \in \sX(S'_1)\underset{\sX(S_1)}{\times}\sX(T_2)$ be its image and $\widetilde{\gamma(y)} \in \sX(T'_2)$ a lift, which exists by the assumption of homogeneity in the Henselian situation. Since $\sX$ is locally almost of finite presentation, one has that 
    \[
        \colim_{U'_2\mbox{ an \'etale nbd. of }s'_2}\sX(U'_2) \overset{\simeq}{\ra} \sX(T'_2),
    \]
    so we can assume that $\widetilde{\gamma(y)}$ is given by some point $u'_2 \in \sX(U'_2)$. 
    
    Finally, we claim that $u'_2 \in \sX(U'_2)$ descends to a point of $x'_2 \in \sX(S'_2)$. Indeed, since $\sX$ is an \'etale sheaf we only need to show that there are isomorphism $\sigma: p^*_1(u'_2) \simeq p^*_2(u'_2)$ where $p^*_i:\sX(U'_2) \ra \sX(U'_2\underset{S'_2}{\times}U'_2)$ (for $i=1,2$), and higher isomorphisms witnessing the compatibilities for $\sigma$.

    Notice that the morphism $\sX(U'_2) \ra \sX(T'_2) \ra \sX(S'_1)\underset{\sX(S_1)}{\times}\sX(T_2)$ factors as follows:
    \begin{equation}
        \label{eq:factorization-for-descent}
        \begin{tikzcd}
            \sX(S'_2) \ar[r] \ar[d] & \sX(U'_2) \ar[r] \ar[d] & \sX(T'_2) \ar[d] \\
            \sX(S'_1)\underset{\sX(S_1)}{\times}\sX(S_2) \ar[r] & \sX(S'_1)\underset{\sX(S_1)}{\times}\sX(U_2) \ar[r] & \sX(S'_1)\underset{\sX(S_1)}{\times}\sX(T_2)
        \end{tikzcd}
    \end{equation}
    for some \'etale morphism $U_2 \ra T_2$. Consider the diagram
    \begin{equation}
        \label{eq:descent-isomorphism}
        \begin{tikzcd}
            \sX(U'_2) \ar[d,"\alpha"'] \ar[r,yshift=1.5pt,"p^*_1"] \ar[r,yshift=-1.5pt,"p^*_2"'] & \sX(U'_2\underset{S'_2}{\times}U'_2) \ar[d,"\alpha_2"] \\
            \sX(S'_1)\underset{\sX(S_1)}{\times}\sX(U'_2) \ar[r,yshift=1.5pt,"p^*_1"] \ar[r,yshift=-1.5pt,"p^*_2"'] & \sX(S'_1)\underset{\sX(S_1)}{\times}\sX(U'_2\underset{S'_2}{\times}U'_2)
        \end{tikzcd}    
    \end{equation}
    we obtain isomorphisms $\sigma': p^*_1(\alpha(u'_2)) \simeq p^*_2(\alpha(u'_2))$ since $\alpha(u'_2)$ comes from a point in $\sX(S'_1)\underset{\sX(S_1)}{\times}\sX(S_2)$ and the commutativity of (\ref{eq:descent-isomorphism}) gives $\alpha\circ p^*_1(u'_2) \simeq \alpha\circ p^*_2(u'_2)$. Finally, assumption $(H^{\rm P}_1)$ gives that $\alpha_2$ is fully faithful, so one obtains the isomorphism $\sigma: p^*_1(u'_2) \simeq p^*_2(u'_2)$.
    
    The implication (2) $\Rightarrow$ (1) is standard and we leave it to the reader.
\end{proof}

\begin{parag}
\label{parag:summary-cotangent-complex-conditions}
We summarize all the conditions regarding the cotangent complex.
\begin{enumerate}
    \item[${^{\rm Pro}(v)}_{\rm all}$] $\sX$ admits a pro-cotangent complex at every point $(S \overset{x}{\ra}\sX) \in \Schaff_{/\sX}$;
    \item[${^{\rm Pro}(v)}_{\rm ec}$] $\sX$ admits a pro-cotangent complex at every point $(S \overset{x}{\ra}\sX) \in \Schaffconv_{/\sX}$;
    \item[${^{\rm Pro}(v)}_{\rm ft}$] $\sX$ admits a pro-cotangent complex at every point $(S \overset{x}{\ra}\sX) \in (\Schaffconvft)_{/\sX}$;
    \item[$(v)_{\rm all}$] $\sX$ admits a cotangent complex at every point $(S \overset{x}{\ra}\sX) \in \Schaff_{/\sX}$;
    \item[$(v)_{\rm ec}$] $\sX$ admits a cotangent complex at every point $(S \overset{x}{\ra}\sX) \in \Schaffconv_{/\sX}$;
    \item[$(v)_{\rm ft}$] $\sX$ admits a cotangent complex at every point $(S \overset{x}{\ra}\sX) \in (\Schaffconvft)_{/\sX}$.
\end{enumerate}
Moreover, when the object $T^*_x\sX$ belongs to $\Pro(\QCoh(S)^-)_{\rm laft}$ (resp.\ $\QCoh(S)^{\rm aperf, -}$\footnote{The subcategory of almost perfect objects in $\QCoh(S)^-$.}) we decorate the conditions above as follows ${^{\rm Pro}(v)}^{\rm laft}_{\rm C}$ (resp.\ $(v)^{\rm laft}_{\rm C}$), where $C\in \{\rm all, ec, ft\}$.
\end{parag}

\subsection{Deformation theory}
\label{subsec:def-thy-bootstrap}

The following is the main technical input in bootstraping deformation theory from the conditions of the \hyperref[thm:main-result]{Main Theorem}.

\begin{lem}[Main Lemma]
\label{lem:main-lemma}
Let $p:\sX \ra S$ be a prestack over $S$ an excellent scheme satisfying:
\begin{enumerate}
    \item $\sX$ is a \'etale sheaf;
    \item $\sX$ is locally almost of finite presentation;
    \item $\sX$ is integrable;
    \item[$(iv)_{\rm triv}$] $\sX$ is $W^{\rm lA}_{\rm triv.}$-homogeneous;
    \item[$(v)_{\rm ft}$] $\sX$ admits a cotangent complex at every $(S \overset{x}{\ra} \sX) \in (\Schaffconvft)_{/\sX}$;
    \item[(viii)] the diagonal morphism $\Delta_{\sX/S}: \sX \ra \sX\underset{S}{\times}\sX$ is $n$-geometric for some $n \geq 0$.
\end{enumerate}
Then $\sX$ is $W_{\rm all}$-homogeneous.
\end{lem}

\begin{proof}
First, since we assumed the diagonal is geometric Lemma \ref{lem:H1-for-laft-affines} implies that $\sX$ satisfies $(H^{W_{\rm all}}_1)$. So by Lemma \ref{lem:H2-for-Henselian-schemes} we are reduced to checking that the morphism
\[
    \sX(T'_2) \ra \sX(T'_1)\underset{\sX(T_1)}{\times}\sX(T_2)
\]
is essentially surjective for every pushout diagram as in (\ref{eq:Henselian-pushout-square}). In other words we want to complete the dashed arrow in the following diagram
\[
    \begin{tikzcd}
     & T'_1 \ar[rd] \ar[rrrd,bend left =50] & & & \\
     T_1 \ar[ru,hook,"\jmath"] \ar[dr,"f"] & & T'_2 \ar[rr,dashed] & & \sX \\
     & T_2 \ar[ru,hook] \ar[rrru,bend right=50] & & & 
    \end{tikzcd}
\]
Consider the restriction of $x_k:\Spec\,k \ra T_2 \ra \sX$ to the closed point of $T_2$. By Proposition \ref{prop:approximate-chart}\footnote{Since $\Spec\,k \ra \classical{S}$ is of finite type, notice that there exists an ideal $\fm_x \subset H^0(R)$ such that $H^0(R) \ra k$, factors as $H^0(R) \ra H^0(R)/\fm_x \ra k$ with $H^0(R)/\fm_x \ra k$ a finitely generated extension.} there exists a factorization
\[
    \Spec\,k \overset{x'_k}{\ra} \Spec(B) \ra \sX
\]
such that $H^{-1}(T^*_{x'_k}(\Spec(B)/\sX))=0$. Thus, by Proposition \ref{prop:smooth-chart} we can improve the factorization to
\[
    \Spec\,k \overset{x''_k}{\ra} \Spec(C) \overset{x_{C}}{\ra} \sX
\]
where $x_C$ is formally smooth. Now we notice that since $\sX \ra S$ and $\Spec(C) \ra S$ are locally almost of finite type, so is $x_C$\footnote{Indeed, a morphism is locally almost of finite presentation if and only if it is locally of finite generation to order $n$, for every $n \geq 0$ and the result follows from \cite[Proposition 4.2.3.3 (3)]{SAG}.}. Consider $\tilde{T_2}:= T_2 \underset{\sX}{\times}\Spec(C)$, which we know is $(n-1)$-geometric, since $\sX$ has $n$-geometric diagonal, then the pullback morphism
\[
    \pi_2: \tilde{T_2}:= T_2 \underset{\sX}{\times}\Spec(C) \ra T_2
\]
is formally smooth and locally almost of finite presentation, hence smooth by Proposition \ref{prop:differentially-smooth-is-smooth}. Thus, up to passing to an \'etale cover of $\Spec\,k$, which one can lift to an \'etale cover of $T_2$\footnote{Notice that since $\sX$ is a sheaf for the \'etale topology it is enough to solve the lifting problem for an \'etale cover.}, Lemma \ref{lem:Henselian-lift-against-smooth-geometric-stacks} gives a lift
\[
    \begin{tikzcd}
    \Spec\,k \ar[r] \ar[d] & \tilde{T_2} \ar[d,"\pi_2"] \\
    T_2 \ar[r] \ar[ru,dashed] & T_2
    \end{tikzcd}
\]

Let $\tilde{T_1}:=T_1 \underset{\sX}{\times}\Spec(C)$, and $\tilde{T'_i}:=T'_i \underset{\sX}{\times}\Spec(C)$ for $i=1,2$, the morphism $T'_1 \ra \Spec(C)$ is given by a lift of the following diagram
\[
    \begin{tikzcd}
    T_1 \ar[r] \ar[d] & \Spec(C) \ar[d] \\
    T'_1 \ar[r] \ar[ru,dashed] & \sX
    \end{tikzcd}
\]
which exists, since the right vertical morphism is formally smooth and $T_1 \hra T'_1$ is a nilpotent embedding. We denote by $\tilde{f}:\tilde{T_1} \ra \tilde{T_2}$ and $\tilde{\jmath}:\tilde{T_1} \ra \tilde{T'_1}$ the lifts of $f$ and $\jmath$, respectively.

Now consider $Z_2:=\tilde{f}(\tilde{T_1}\backslash \tilde{\jmath}^{-1}(T'^{\rm sm.}_1))$, where $\tilde{T'_1}^{\rm sm.}$ denote the smooth locus of the morphism $\tilde{T'_1} \ra T_1$\footnote{Notice that $\tilde{T_2} \ra T_2$ and $\tilde{T_1} \ra T_1$ are already smooth.}. Then we pass to the non-empty opens $\bar{T_2}:=\tilde{T_2}\backslash Z^{\rm c}_2$, where $Z^{\rm c}_2$ denotes the closure of $Z_2$, $\bar{T_1}:=\bar{T_2}\underset{\tilde{T_2}}{\times}\tilde{T_1}$ and $\bar{T'_1}:=\tilde{\jmath}(\bar{T_1})$. We claim that the section $s_2:T_2 \ra \tilde{T_2}$ factors through $\bar{T_2}$, since $Z^{\rm c}_2$ doesn't contain any point above $x''_C$, which follows from $Z_2$ not containing any point above $\Spec(C)$. 
Thus, one has a section $s_2:T_2 \ra \bar{T_2}$ which by restriction gives a section $s_1:T_1 \ra \bar{T_1}$ that extends to section $s'_1:T'_1 \ra \bar{T'_1}$. Consider the diagram 
\[
    \begin{tikzcd}[row sep={40,between origins}, column sep={40,between origins}]
      & \bar{T_1} \ar[hookrightarrow]{rr}\ar{dd}\ar{dl} & & \bar{T'_1} \ar{dd}\ar{dl} \\
    \bar{T_2} \ar[crossing over]{rr} \ar{dd} & & \bar{T'_2} \\
      & T_1  \ar[hookrightarrow]{rr}[xshift = -10pt]{\jmath} \ar{dl}[swap]{f} & &  T'_1 \ar{dl} \\
    T_2 \ar{rr} && T'_2 \ar[from=uu,crossing over]
    \end{tikzcd}
\]
where the top square is a pushout, i.e.\ defining $\bar{T'_2}$, the bottom square is also a pushout, and each side face is a pullback. Notice that we can glue the sections $s_1$, $s_2$ and $s'_1$ to obtain a section $s'_2:T'_2 \ra \bar{T'_2}$. Since $\Spec(C)$ is $W_{\rm all}$-homogeneous one has a dashed arrow making the diagram
\[
    \begin{tikzcd}
     & \bar{T'_1} \ar[rd] \ar[rrrd,bend left =50] & & & \\
     \bar{T_1} \ar[ru,hook,"\jmath"] \ar[dr,"f"] & & \bar{T'_2} \ar[rr,dashed,"x'_{2,C}"] & & \Spec(C) \\
     & \bar{T_2} \ar[ru,hook] \ar[rrru,bend right=50] & & & 
    \end{tikzcd}
\]
commutative. Thus, the composite $T'_2 \overset{s'_2}{\ra}\bar{T'_2} \overset{x'_{2,C}}{\ra} \Spec(C) \ra \sX$ solves our lifting problem and we are done.
\end{proof}

\section{Ingredients}

In this section we collect the main technical results that go into Lurie's version of the representability theorem. Our purpose here is to give enough details to convince the reader that the necessary ingredients work in our situation without redoing all the proofs.

\subsection{Formal Charts}

In any proof of this theorem some form of the following construction seems very relevant.

\begin{construction}
\label{cons:formal-nbd-of-a-point}
Let $y:\Spec\,k \ra \sY$ denote a field point in a prestack $\sX$. Assume that
\begin{condlist}
    \item the prestack $\sY$ satisfies 
        \begin{enumerate}
            \setcounter{enumi}{3}
            \item $\sY$ is $W^{\rm lA}_{\rm all}$-homogeneous,
            \item[$(v)_{\rm Art.}$] $\sY$ admits a cotangent complex at every local Artinian scheme of finite type,
            \item $\sY$ is convergent;
        \end{enumerate}
    \item $\sY$ is formally complete along $f$, i.e.\ $f_{\rm dR}:\Spec\,k_{\rm dR} \ra \sY_{\rm dR}$ is an isomorphism;
    \item one is given a morphism $\alpha:\sF \ra T^*(\Spec\,k/\sY)$, where $\sF$ is perfect of Tor-amplitude $\leq 0$ and $\Cofib \alpha$ is $1$-connective and almost perfect.
\end{condlist}
Then there exists a sequence
\[
    \begin{tikzcd}
    \Spec\,k = S_0 \ar[r] \ar[rrrd,"y_0"'] & S_1 \ar[r] \ar[rrd,"y_1"] & \cdots \ar[r] & S_n \ar[r] \ar[d,"y_n"] & \cdots \\
     & & & \sY & 
    \end{tikzcd}
\]
and morphisms $\alpha_n: \sF_n \ra T^*(S_n/\sY)$, where $\sF_n$ is perfect of Tor-amplitude $\leq 0$ and $\Cofib \alpha_n$ is $1$-connective. 

Consider
\[
    U := {^{\rm conv}\left(\colim_{n \geq 0}S_n\right)}
\]
the convergent completion of $\colim_{n \geq 0}S_n$. One has a canonical factorization
\[
    \Spec\,k \overset{y'}{\ra} U \overset{y''}{\ra} \sY.  
\]
\end{construction}

\begin{proof}
Assume $n \geq 0$ and that we have constructed $y_n:S_n \ra \sY$ and a morphism $\sF_n \ra T^*(S_n/\sY)$ in $\QCoh(S_n)$ satisfying the required conditions.

Let $S_{n+1}$ be the square-zero extension of $S_n$ by $\Cofib(\alpha_n)[-1]$, i.e.\ 
\[
    S_{n+1}:= S_n\underset{(S_n)_{\Cofib (\alpha_n)}}{\sqcup}S_n
\]
where one of the morphisms $(S_{n})_{\Cofib(\alpha_n)} \ra S_n$ is determined by the map $\gamma_n:T^*S_n \ra T^*(S_n/\sY) \ra \Cofib(\alpha_n)$. Since the restriction of $\gamma$ to $T^*_{f_n}\sY$ is null-homotopic, one obtains a factorization of $y_n$ as follows
\[
    S_n \overset{g_n}{\ra} S_{n+1} \overset{y_{n+1}}{\ra} \sY.
\]
Let
\[
    \begin{tikzcd}
    S_n \ar[r,"g'_n"] \ar[d,"g_n"] & S'_{n} \ar[ld,"g''_n"] \\
    S_{n+1}
    \end{tikzcd}
\]
be the diagram where $S'_n$ is the square-zero extension of $S_n$ by $T^*(S_n/S_{n+1})$, and the existence of the factorization of $g_n$ follows from the same argument as before. The canonical maps
\[
    g^*_n\sO_{S_{n+1}} \ra (g'_n)^*\sO_{S'_{n}} \ra \sO_{S_n}
\]
give rise to the following fiber sequence
\[
    (g'_n)^*\Cofib((g''_n)^*\sO_{S_{n+1}} \ra \sO_{S'_n}) \overset{\epsilon_{g_n}}{\ra} \Cofib( g^*_n\sO_{S_{n+1}} \ra \sO_{S_n}) \ra \Cofib((g'_n)^*\sO_{S'_{n}} \ra \sO_{S_n})
\]
where one has 
\[
    (g'_n)^*\Cofib((g''_n)^*\sO_{S_{n+1}} \ra \sO_{S'_n}) \simeq T^*(S_n/S_{n+1}) \;\;\; \mbox{and} \;\;\; \Cofib( g^*_n\sO_{S_{n+1}} \ra \sO_{S_n}) \simeq \Cofib(\alpha_n)
\]
and $\epsilon_{g_n}$ is the morphism constructed in \cite[Construction 7.4.3.10]{HA}. Notice that $g_n$ is $1$-connective, since $\Cofib(\alpha_n)$ is $1$-connective. Thus, \cite[Theorem 7.4.3.12]{HA} gives that $\epsilon_{g_n}$ is $2$-connective. In particular, one obtains that
\[
    \tau^{\geq -1}(T^*(S_n/S_{n+1})) \overset{\simeq}{\ra} \tau^{\geq -1}(\Cofib(\alpha_n)).
\]
So the morphism
\[
    \left.\gamma_n\right|_{T^*_{f_n}\sY}: T^*_{f_n}\sY \ra T^*(S_n/\sY) \ra \Cofib(\alpha_n)
\]
is null-homotopic, since it induces the zero morphism on $H^0$. This implies that one has a factorization of $\gamma_n$ as follows
\[
    T^{*}(S_n/\sY) \ra T^*(S_n/S_{n+1}) \overset{\epsilon_{g_n}}{\ra} \Cofib(\alpha_n),
\]
which by passing to the corresponding fibers gives
\[
    g^*_nT^*(S_{n+1}/\sY) \ra \sF_n \overset{\nu}{\ra} \Fib(\epsilon_{g_n}).
\]
As we already notice, since $\epsilon_{g_n}$ is $2$-connective, one has that the morphism $\nu$ is null-homotopic. Thus, the composite
\[
    \begin{tikzcd}
    \sF_n \ar[r,"\alpha_n"] & T^*(S_n/\sY) \ar[r] 
    \ar[rd] & \Fib(\epsilon_{g_n}) \ar[d] \\
    & & T^*(S_n/S_{n+1})
    \end{tikzcd}
\]
is null-homotopic, where we notice that the composite factors through $\Fib(\epsilon_{g_n})$ by definition. Thus, one obtains a factorization of $\alpha_n$ as follows:
\[
    \sF_n \overset{\alpha'_n}{\ra} g^*_nT^*(S_{n+1}/\sY) \ra T^*(S_n/\sY).
\]
By passing to the cofibers of the above composite one obtains a fiber sequence
\[
    \Cofib(\alpha'_n) \ra \Cofib(\alpha_n) \overset{\beta}{\ra} T^*(S_n/S_{n+1})
\]
whose truncation to degrees $-1$ and above agrees with the truncation of
\[
    g^*_n T^*(S_n/\sY) \ra T^*(S_n/\sY) \ra T^*(S_n/S_{n+1})
\]
by \cite[Theorem 7.4.3.12]{HA}. In particular, one obtains that the induced morphism $H^{-1}(\beta)$ is surjective, so $\Cofib(\alpha'_n)$ is $1$-connective. 

By \cite[Lemma 18.2.5.4]{SAG} one can lift $\alpha'_n$ to a morphism
\[
    \widetilde{\alpha'_n}:\sF_{n+1} \ra T^*(S_{n+1}/\sY),
\]
where $\sF_{n+1} \in \QCoh(S_{n+1})$ is perfect of Tor-amplitude $\leq 0$
\end{proof}

The following collects some properties of Construction \ref{cons:formal-nbd-of-a-point}, most of them are proved by a careful examination of the construction.

\begin{prop}
\label{prop:properties-formal-chart}
Let $y:\Spec\,k \ra \sY$ be a morphism satisfying conditions 1-3) from Construction \ref{cons:formal-nbd-of-a-point}, then
\begin{resultlist}
    \item $U$ satisfies
        \begin{enumerate}
            \setcounter{enumi}{3}
            \item $U$ is $W^{\rm lA}_{\rm all}$-homogeneous,
            \item[$(v)_{\rm Art}$] $U$ admits a cotangent complex at every local Artinian affine scheme of finite type,
            \item $U$ is convergent;
        \end{enumerate}
    \item the canonical morphism $T^*(\Spec\,k/U) \ra T^*(\Spec\,k/\sY)$ identifies with the morphism $\alpha:\sF \ra T^*(\Spec\,k/\sY)$;
    \item the morphism $U \ra \sY$ is locally almost of finite presentation;
    \item suppose further that $\sF \simeq 0$, then $\sY$ is representable by a formal thickening of $\Spec\,k$.
\end{resultlist}
\end{prop}

\begin{proof}
All the arguments of \S 18.2.5 \cite{SAG} go through in our context. Specifically, we have: (a) and (b) are Lemma 18.2.5.6, (c) is Lemma 18.2.5.9, and (d) is Lemma 18.2.5.19.
\end{proof}

\begin{cor}
\label{cor:formal-charts}
Let $y:\Spec\,k \ra \sY$ denote a field point in a prestack $\sX$. Assume that
\begin{condlist}
    \item the prestack $\sY$ satisfies 
        \begin{enumerate}
            \setcounter{enumi}{3}
            \item $\sY$ is $W^{\rm lA}_{\rm all}$-homogeneous,
            \item[$(v)_{\rm Art}$] $\sY$ admits a cotangent complex at every local Artinian scheme of finite type,
            \item $\sY$ is convergent;
        \end{enumerate}
    \item $\sY$ is formally complete along $f$, i.e.\ $f_{\rm dR}:\Spec\,k_{\rm dR} \ra \sY_{\rm dR}$ is an isomorphism;
    \item one is given a morphism $\alpha:\sF \ra T^*(S/\sY)$, where $\sF$ is perfect of Tor-amplitude $\leq 0$ and $\Cofib \alpha$ is $1$-connective and almost perfect.
\end{condlist}
Then, there exists a factorization
\[
    \Spec\,k \overset{y'}{\ra} U \overset{y''}{\ra} \sY
\]
where $U \simeq \Spf(A)$ is an affine formal scheme, the morphism $y':\Spf(k) \ra \Spf(A)$ is a formal thickening, $y''$ is locally almost of finite presentation and $\alpha$ can be identified with the canonical morphism $T^*_{y'}(U/\sY) \ra T^*_y(\Spec\,k/\sY)$.
\end{cor}

\subsection{Approximate charts}

In this section we go through the construction due to J.\ Lurie that enlarges the morphism from a field point to a morphism from an affine scheme that is almost formally smooth over $\sX$, i.e.\ the relative cotangent complex vanishes only in degree $-1$, instead of all negative degrees. Both the statement and proof of the following proposition are tiny variations on \cite[Proposition 18.3.1.1]{SAG}.

\begin{prop}
\label{prop:approximate-chart}
Let $p:\sX \ra S$ be a prestack, assume that 
\begin{condlist}
    \item $\sX$ satisfies the following subset of conditions from \hyperref[thm:main-result]{Main Theorem}:
        \begin{enumerate}
        \setcounter{enumi}{1}
            \item $\sX$ is locally almost of finite presentation;
            \item $\sX$ is integrable;
            \item[$(iv)_{\rm triv}$] $\sX$ is $W^{\rm lA}_{\rm triv.}$-homogeneous;
            \item[$(v)_{\rm ft}$] for every $(T \overset{x}{\ra} \sX) \in (\Schaffconvft)_{/\sX}$, $\sX$ admits a cotangent complex at $x$;
            \item[(vi)] $\sX$ is convergent;
        \end{enumerate}
    \item $S$ is Grothendieck affine scheme.
\end{condlist}
Consider a field point $x:\Spec\,k \ra \sX$ such that $k$ is a finitely generated field extension of some residue field of $R$. Then there exists a factorization
\[
    \Spec\,k \overset{x_B}{\ra} \Spec\,B \ra \sX
\]
where 
\begin{itemize}
    \item $\Spec\,B \ra S$ is almost of finite presentation;
    \item $H^{-1}(T^*_{x_B}(\Spec\,B/\sX)) = 0$.
\end{itemize}
\end{prop}

\begin{proof}

\textit{Step 1 (formal completion at the point):}
The composite $\Spec\,k \ra \sX \ra S$ gives us the fiber sequence
\[
    T^*_{x}\sX \ra T^*(\Spec\, k /S) \ra T^*(\Spec\, k / \sX).
\]
By Proposition \ref{prop:properties-cotangent-complex} $T^*(\Spec k /S)$ is almost perfect and $T^*_{x}\sX$ is almost perfect by assumption, so we have that $T^*(\Spec k / \sX)$ is almost perfect.

Consider the formal completion of $x:\Spec\, k \ra \sX$
\[
    \sX^{\wedge}_x := \sX \underset{\dR{\sX}}{\times}\dR{\Spec\, k}
\]
where $\dR{\sX}$ denotes the de Rham prestack associated to $\sX$ (see \cite[Chapter 4, \S 1.1.1]{GRII}). Recall that for any prestack $\sZ$ its de Rham prestack $\dR{\sZ}$ is convergent, infinitesimally cohesive and admits a cotangent complex which is the zero object. Thus, we have
\[
    T^*(\sX^{\wedge}_x/\sX) \simeq p^*_1(T^*(\dR{\Spec\,k}/\dR{\sX})) \simeq 0
\]
where $p_1:\sX^{\wedge}_x \ra \dR{\Spec\,k}$ is the canonical projection. So the composite $\Spec\,k\overset{\hat{x}}{\ra} \sX^{\wedge}_{x} \ra \sX$ gives the sequence
\[
    T^*_{x}\sX \overset{\simeq}{\ra} T^*_{\hat{x}}(\sX^{\wedge}_{x}) \ra T^*(\sX^{\wedge}_x/\sX),    
\]
from where we conclude that $T^*_{\hat{x}}(\sX^{\wedge}_{x})$ exists and is isomorphic to $T^*_{x}\sX$. Thus, $T^*(\Spec\,k/\sX^{\wedge}_{x})$ also exists. 

\textit{Step 2 (formal chart):} 
Now we pick $\alpha:\sF \ra T^*(\Spec\,k/\sX^{\wedge}_{x})$ where $\sF$ is perfect and has Tor-amplitude $\leq 1$ and such that $\Cofib(\alpha)$ is $1$-connective\footnote{Notice $T^*(\Spec\,k/\sX^{\wedge}_{x})$ is perfect to order $0$, since it is almost perfect, so by \cite[Corollary 2.7.2.2]{SAG} one can always find $\sF$ and $\alpha$ as claimed. Actually, notice that in the statement of \cite[Corollary 2.7.2.2]{SAG} one can actually find $P \ra M$ with fiber $(n+1)$-connective, instead of just $n$-connective, with the proof given in \emph{loc.\ cit.}.}, i.e.\ $\Cofib(\alpha) \in \Mod^{\leq -1}_k$. Then, Corollary \ref{cor:formal-charts} gives a factorization
\[
    \Spec\,k \overset{x'}{\ra} \Spf(A) \overset{x''}{\ra} \sX.
\]

\textit{Step 3 (affine chart):}
Notice that since $\Spf(A)^{\rm red} \simeq \Spec\,k$, the ring $A$ is Noetherian and one has an ideal $I\subset A$ such that $A/I \simeq k$. Since, $\Spf(A) \simeq \colim_{n\geq 0}\Spec(A/I^n)$ as functor of points, and $\sX$ is convergent and integral one has
\[
    \Maps(\Spec\,A,\sX) \overset{\simeq}{\ra} \Maps(\Spf(A),\sX) \overset{\simeq}{\ra} \lim_{n\geq 0}\sX(\Spec(A/I^n)).
\]
Thus, one has a factorization
\[
    \Spec\,k \overset{x_A}{\ra} \Spec\,A \overset{\bar{x}}{\ra} \sX.
\]
Notice that by Proposition \ref{prop:properties-formal-chart} (2) the associated fiber sequence
\[
    x^*_a T^*(\Spec\, A /\sX) \ra T^*(\Spec k/\sX) \ra T^*(\Spec k/\Spec\, A)
\]
has first term perfect of Tor-amplitude $\leq 0$ and last term $1$-connective and almost perfect. In particular, one has that 
\[
    H^i(x^*_a T^*(\Spec\, A /\sX)) = 0
\]
for $i < 0$, since $k$ is a field.

We would be done if $\Spec\,A \ra S$ was almost of finite presentation, however, this need not be the case.

\textit{Step 4 (restriction to a classical approimate chart):}
The composite $\Spec\,k \overset{x_{H^0(A)}}{\ra} \classical{\Spec(A)} \overset{\imath_{A}}{\ra} \Spec(A) \ra \sX$, where we notice that $\classical{\Spec(A)}= \Spec(H^0(A))$, gives a fiber sequence
\[
    \imath^*_A T^*_{x_A}(\Spec\, A /\sX) \ra T^*_{x_{H^0(A)}}(\Spec\, H^0(A)/\sX) \ra T^*(\Spec\, H^0(A)/\Spec\, A),
\]
where each relative cotangent complex exists from our hypothesis. Since $\Spec\, H^0(A) \ra \Spec\, A$ induces an isomorphism on the classical level, $T^*(\Spec\, H^0(A)/\Spec\, A)$ is $2$-connective, so 
\[
    \tau^{-1}(\imath^*_A T^*(\Spec\, A /\sX)) \overset{\simeq}{\ra} \tau^{\geq -1}(T^*(\Spec\, H^0(A)/\sX)).
\]
In particular, $H^-1(T^*_{x_{H^0(A)}}(\Spec\, H^0(A)/\sX)) = 0$.

\textit{Step 5 (making the chart almost of finite type):} This is arguably the hardest step of the proof.

\textit{Step 5 (i) (approximate the classical truncation by a completion of a finite type scheme):}

Let $\fm \subset H^0(A)$ denote the maximal ideal, such that $k \simeq H^0(A)/\fm$. By classical Noetherian approximation we can find $A' \subset H^0(A)$ a finitely generated $H^0(R)$-subalgebra such that
\begin{itemize}
    \item the fraction field of $A'/\fm \cap A'$ is isomorphic to $k$;
    \item $A'$ has a basis for the vector space $\fm/\fm^2$.
\end{itemize}

Let $\fp:= \fm \cap A'$ and consider the localization $A'_{\fp}$. The morphism $A' \ra H^0(A)$ naturally factors as
\[
    A' \ra A'_{\fp} \overset{v}{\ra} H^0(A),
\]
where $v$ is surjective and induces a surjection on $H^0$ of the cotangent complexes. Since $H^0(A)$ is complete with respect to $\fm$ the morphism $v$ extends to a \emph{surjective} $\hat{v}:\hat{A'_{\fp}} \ra H^0(A)$, where $\hat{A'_{\fp}}$ denotes the completion with respect to $\fp$. In other words the morphism $\Spec H^0(A) \ra \Spec(\hat{A'_{\fp}})$ is a closed embedding, hence it is almost of finite presentation.

\textit{Step 5 (ii) (Popescu's theorem to approximate the completion):}
Thus, one has a diagram
\[
    \begin{tikzcd}
    \Spec\, H^0 (A) \ar[r,"\hat{v}"] & \Spec(\hat{A'_{\fp}}) \ar[r,"v'"] & \Spec\, A'_{\fp} \ar[r] & \Spec\,A' \ar[d] \\
    & & & H^0(R)
    \end{tikzcd}
\]
where $v'$ is geometrically regular, since $A'$ is a Grothendieck ring as a finitely generated algebra over an excellent ring $H^0(R)$. Thus, one can apply Theorem \ref{prop:derived-Popescu} to obtain 
\[
    \Spec(\hat{A'_{\fp}}) \overset{\simeq}{\ra} \lim_{I}\Spec\,B_i
\]
where $\Spec\,B_i \ra \Spec\,A'$ is smooth.

\textit{Step 5 (iii) (classical truncation is determined at finite level):}
Since $\Spec\, H^0(A) \ra \Spec(\hat{A'_{\fp}})$ is almost of finite presentation, by Theorem 4.4.2.2 and Proposition 4.6.1.1 \cite{SAG} one can find $i_0 \in I$ and $h:\Spec\, C_{i_0} \ra \Spec\, B_{i_0}$, such that $\Spec\, C_{i_0}$ is $1$-truncated and $h$ is of finite generation to order $2$, and we have
\[
    \tau^{\leq 1}(\Spec(\hat{A'_{\fp}}) \underset{\Spec\, B_{i_0}}{\times} \Spec\, C_{i_0}) \simeq \Spec\, H^0(A).
\]
This is an analogue of Noetherian approximation, but it needs a bit more of theory in the derived setting since many of the data for derived rings involves infinitely many coherence (see the introduction of \cite[\S 4.4]{SAG} for a discussion). For $i \geq i_0$ we define $\Spec\,C_i := \tau^{\leq 1}(\Spec\, B_i \underset{\Spec\, B_{i_0}}{\times} \Spec\, C_{i_0})$, such that we canonically have
\[
    \Spec\, H^0(A) \simeq \tau^{\leq 1}(\Spec(\hat{A'_{\fp}}) \underset{\Spec\, B_{i_0}}{\times} \Spec\, C_{i_0}) \simeq \tau^{\leq 1}(\lim_{I^{\rm op}_{\geq i_0}} \Spec\, B_i \underset{\Spec\, B_{i_0} }{\times} \Spec\, C_{i_0} )  \simeq \lim_{I^{\rm op}_{\geq i_0}}\Spec\, C_i.
\]
Since $\sX$ is assumed to be locally almost of finite presentation the morphism $\Spec H^0(A) \ra \sX$ is determined by
\[
    \Spec\, H^0(A) \ra \Spec\, C_{i} \overset{x'_{C_i}}{\ra} \sX
\]
for some $i\in I$. We claim that this didn't spoil the vanishing of the cotangent complex in degree $1$, i.e.\ 
\begin{equation}
    \label{eq:vanishing-of-cotangent-complex-of-approximate-chart}
    H^{-1}(T^*_{x_{C_i}}(\Spec\, C_i / \sX)) = 0,
\end{equation}
where we denote $x_{C_i}:\Spec k \ra \Spec\,C_i$

\textit{Step 6 (vanishing of cotangent complex argument):} 

The composite $\Spec H^0(A) \ra \Spec C_{i} \ra \sX$ gives a fiber sequence
\[
    T^*_{x_{C_i}}(\Spec\, C_i / \sX) \ra T^*_{x_{H^0(A)}}(\Spec\, H^0(A) / \sX) \ra T^*_{x_{H^0(A)}}(\Spec\, H^0(A) / \Spec\, C_i),
\]
since $H^{-1}(T^*_{x_{H^0(A)}}(\Spec\, H^0(A) / \sX)) = 0$, it is enough to check that $H^{-2}(T^*_{x_{H^0(A)}}(\Spec\, H^0(A) / \Spec\,C_i))$ vanishes. Consider the pullback
\begin{equation}
    \label{eq:pullback-defining-C'_i}
    \begin{tikzcd}
    \Spec\,C'_i \ar[r] \ar[d] & \Spec\, \hat{A'_{\fp}} \ar[d] \\
    \Spec\, C_i \ar[r] & \Spec\, B_i
    \end{tikzcd}
\end{equation}
defining $\Spec\,C'_i$. Notice that $\tau^{\geq 1}(\Spec\,C'_i) \simeq \Spec\,H^0(A)$ by construction. Thus, the composite $\Spec\,H^0(A) \ra \Spec\,C'_i \ra \Spec\,C_i$ gives a short-exact sequence
\[
    H^{-2}(T^*_{x_{C'_i}}(\Spec\, C'_i / \Spec\,C_i)) \ra H^{-2}(T^*_{x_{H^0(A)}}(\Spec\, H^0(A) / \Spec\,C_i)) \ra H^{-2}(T^*_{x_{H^0(A)}}(\Spec\, H^0(A) / \Spec\,C'_i)),
\]
where the third terms vanishes by the connectivity of $\Spec\,H^0(A) \ra \Spec\,C'_I$, so it is enough to check that the first term vanishes. From (\ref{eq:pullback-defining-C'_i}) one has $T^*_{x_{C'_i}}(\Spec\, C'_i / \Spec\,C_i) \simeq T^*_{x_{\hat{A'_{\fp}}}}(\Spec\,\hat{A'_{\fp}}/\Spec\,B_i)$ so it is enough to check that $H^{-2}(T^*_{x_{\hat{A'_{\fp}}}}(\Spec\,\hat{A'_{\fp}}/\Spec\,B_i))$ vanishes. The composite $\Spec\,\hat{A'_{\fp}} \ra \Spec\,B_i \ra \Spec\,A'$ gives the short exact sequence
\[
    H^{-2}(T^*_{x_{\hat{A'_{\fp}}}}(\Spec\,\hat{A'_{\fp}}/\Spec\,A')) \ra H^{-2}(T^*_{x_{\hat{A'_{\fp}}}}(\Spec\,\hat{A'_{\fp}}/\Spec\,B_i)) \ra H^{-1}(T^*_{x_{B_i}}(\Spec\,B_i/\Spec\,A')),
\]
and since one has a pullback diagram
\[
    \begin{tikzcd}
    \Spec\,k \ar[r,"x_{\hat{A'_{\fp}}}"] \ar[d] & \Spec\,\hat{A'_{\fp}} \ar[d] \\
    \Spec\,k \ar[r,"x_{A'}"] & \Spec\,A'
    \end{tikzcd}
\]
one has $T^*_{x_{\hat{A'_{\fp}}}}(\Spec\,\hat{A'_{\fp}}/\Spec\,A') \simeq T^*(\Spec\,k / \Spec\,k) \simeq 0$ and we are left with showing that $H^{-1}(T^*_{x_{B_i}}(\Spec\,B_i/\Spec\,A'))$ vanishes.

Consider $D$ defined by the pullback
\[
    \begin{tikzcd}
    \Spec\,D \ar[r] \ar[d] & \Spec\,B_i \ar[d] \\
    \Spec\,k \ar[r] & \Spec\,A'
    \end{tikzcd}
\]
Notice that $\Spec\,D \ra \Spec\,k$ is smooth and has a base point $x_D:\Spec\,k \ra \Spec\,D$. Moreover, the defining pullback diagram gives that 
\[
    H^{-1}(T^*_{x_{B_i}}(\Spec\,B_i/\Spec\,A')) \simeq H^{-1}(T^*_{x_D}(\Spec\,D /\Spec\,k)).
\]
However,
\[
    \left(H^{-1}(T^*_{x_D}(\Spec\,D /\Spec\,k))\right)^{\vee} \simeq \Hom_{\QCoh(\Spec\,D)}(H^{-1}(T^*_{x_D}(\Spec\,D /\Spec\,k)),(x_{D})_* k)
\]
is by definition the connected components of the space of lifts
\begin{equation}
    \label{eq:space-of-lifts-of-D}
    \begin{tikzcd}
    \Spec\,D \ar[r,"p"] \ar[d] & \Spec\,k \\
    \Spec\,D'\ar[ru,dashed,"p'"] &
    \end{tikzcd}
\end{equation}
where $\Spec\,D'$ is the square-zero extension of $\Spec\,D$ determined by $(x_{D})_* k$. Given any lift $p'$ the diagram
\[
    \begin{tikzcd}
    \Spec\,D \ar[r,"\id"] \ar[d] & \Spec\,D \ar[d,"p"] \\
    \Spec\,D'\ar[r,"p'"] \ar[ru,dashed] & \Spec\,k
    \end{tikzcd}
\]
has a lift, since $p$ is smooth, i.e.\ the set of isomorphism classes of $p'$ is trivial, that is the set of connected components of the space (\ref{eq:space-of-lifts-of-D}) is trivial. Thus, our cohomology group vanishes and the proof is finished.
\end{proof}

\subsection{Smooth charts}

The great think about approximate charts as the ones obtained in Proposition \ref{prop:approximate-chart} is that one can refine them to actually (formally) smooth charts.

\begin{prop}
\label{prop:smooth-chart}
Let $p:\sX \ra S$ be a prestack, assume that 
\begin{condlist}
    \item $p$ satisfies conditions (ii), $(iv)_{\rm triv}$, $(v)_{\rm ft}$ and (vi) from Theorem \ref{thm:main-result};
    \item we are given a point $x_B:\Spec\,B \ra \sX$, such that for some prime ideal $\fp \subset B$ the restriction $\Spec\,\kappa(\fp) \overset{x_{\kappa(\fp)}}{\ra} \Spec\, B \ra \sX$ to the residue field at $\fp$ satisfies $H^{-1}(T^*_{x_{\kappa(\fp)}}(\Spec\,B/\sX)) = 0$.
\end{condlist}
Then one can find a diagram
\[
    \begin{tikzcd}
    \Spec\,B' \ar[r,dashed] \ar[d] & \Spec\,C \ar[d,"x_{C}"] \\
    \Spec\,B \ar[r,"x_{B}"] & \sX
    \end{tikzcd}
\]
such that 
\begin{resultlist}
    \item $\Spec\,B' \ra \Spec\,B$ is Zariski neighborhood of $\Spec\,\kappa(\fp)$;
    \item $\classical{(\Spec\,B')} \overset{\simeq}{\ra} \classical{(\Spec\,C)}$;
    \item $\Spec\, C \ra \sX$ is formally smooth.
    \item if $\Spec\, B \ra S$ is almost of finite presentation, $\Spec\, C \ra S$ is also almost of finite presentation.
\end{resultlist}
\end{prop}

\begin{proof}
This is \cite[Lemma 7.2.1]{DAG}.
\end{proof}

\subsection{Desingularization}

There are a couple of places where one needs a crucial desingularization result. 

The first is in the construction an approximate chart at a field point of $\sX$, i.e.\ an affine scheme almost of finite type over the base with vanishing relative cotangent complex in degree $-1$. A formal smooth chart can be obtained roughly by considering a formal chart at the formal completion of the point and passing to an actual neighborhood using the integrability assumption on $\sX$, the hard part is to guarantee that this is almost of finite type over the base--Popescu's theorem allows one to refine this actual chart to one almost of finite type, at the expense of only approximating the relative cotangent complex at a single degree. 

The second use is in lifting a point of our prestack $\sX$ only defined in a completion of a point to an actual \'etale neighborhood of that point\footnote{Very much the same as what Artin's approximation and some generalizations of it do in the classical case (cf.\ \cite[Theorem 1.8]{Alper-Mainz}).} (see Step 9 in proof of Proposition \ref{thm:etale-surjection}).


The following is a derived analogue of Popescu's theorem:

\begin{prop}
\label{prop:derived-Popescu}
Consider a morphism $f:T \ra S$ between Noetherian affine schemes, the following conditions are equivalent:
\begin{equivlist}
    \item there exists a filtered diagram $I$, such that $T \simeq \lim_{I^{\rm op}}T_i$ where each $T_i \ra S$ is smooth;
    \item the morphism $f$ is geometrically regular, i.e.\ $f$ is flat and $\classical{f}:\classical{T} \ra \classical{S}$ is geometrically regular in the usual sense (cf.\ \cite[Tag 0382]{stacks-project}).
\end{equivlist}
\end{prop}

\begin{proof}
This is Theorem 3.7.5 in \cite{DAG}.
\end{proof}

\section{Proof of Main Theorem}

\subsection{Properties of \texorpdfstring{$n$}{n}-geometric}

The hardest part of the proof of the Main result is checking that conditions (i-vii) are sufficient for $\sX$ to be an $n$-geometric stack. In this section we check that the conditions of \hyperref{thm:main-result}{Main Theorem} are necessary.

The fact that any $n$-geometric stack is integrable is probably the hardest condition to check.

\begin{prop}
\label{prop:n-geometric-is-integrable}
Let $\sX$ be an $n$-geometric stack, then $\sX$ is integrable.
\end{prop}

\begin{proof}
We pass to a more general statement and prove it by induction. 

Consider $B$ a complete local Noetherian ring $B$ with maximal ideal $\fm$ and $k := B/\fm$ residue field. Let $T = \Spec\, B$, $t^{(n)} := \Spec(B/\fm^{n+1})$ for $n\geq 1$ and $t = \Spec\,k$. Consider $\Schaff_{\rm f.\acute{e}t. in T}$ the category of finite \'etale affine schemes over $T$, because $B$ is a Henselian ring, by \cite[Proposition B.6.5.2]{SAG} one has an equivalence
\[
    \Phi^{\infty}: \Schaff_{\rm \acute{e}t. in t}  \overset{\simeq}{\ra} \Schaff_{\rm f.\acute{e}t. in T}
\]
where $\Schaff_{\rm \acute{e}t. in t}$ is the category of \'etale affine schemes over $t$. Actually by \cite[Proposition B.3.3.7]{SAG} and noticing that $(B/\fm^{n},\fm \cdot B/\fm^{n})$ is a Henselian pair for any $n \geq 2$, one has compatible equivalences
\[
    \Phi^{(n)}: \Schaff_{\rm \acute{e}t. in t} \overset{\simeq}{\ra} \Schaff_{\rm \acute{e}t. in t^{(n)}}
\]
for each $n\geq 1$. We let
\[
    (\Phi^{\infty})^*(\sX): (\Schaff_{\rm \acute{e}t. in t})^{\rm op} \ra \Spc
\]
denote the \'etale sheaf obtained by pulling back $\sX$ through the equivalence $\Phi^{\infty}$\footnote{Notice that this is not simply restricting a point $\sX(T)$ to a point $\sX(t)$, the \emph{inverse} of $\Phi^{\infty}$ is given by a fiber product, but not $\Phi$ itself.}. Simiarly, we let
\[
    (\Phi^{(n)})^*(\sX): (\Schaff_{\rm \acute{e}t. in t})^{\rm op} \ra \Spc
\]
denote the sheaf obtained with the equivalence $\Phi^{(n)}$. Because the representable sheaf $h_T$ in $\Schaff_{\rm f.\acute{e}t. in T}$ is the final object one has an equivalence
\[
    \sX(T) \simeq \Hom_{\mbox{Shv}((\Schaff_{\rm f.\acute{e}t. in T})^{\rm op})}(h_T,\sX) \simeq \Hom_{\mbox{Shv}((\Schaff_{\rm \acute{e}t. in t})^{\rm op})}(h_{t},(\Phi^{\infty})*(\sX)) \simeq (\Phi^{\infty})*(\sX)(t)
\]
and, by definition 
\[
    \Phi^{(n)})(\sX)(t) \simeq \sX(t^{(n)})
\]
for any $n\geq 1$. Thus, to check that $\sX(T) \ra \lim_{n \geq 1}\sX(t^{(n)})$ is an equivalence, it is enough to check that
\[
    (\Phi^{\infty})^*(\sX)(t) \ra \lim_{n \geq }(\Phi^{(n)})^*(\sX)(t).
\]

Let $\sU \ra \sX$ be a smooth atlas of $\sX$. Notice that we have a factorization
\[
    \left|(\Phi^{\infty})^*(\sU^{\bu})\right|_{\rm Stk} \overset{\varphi}{\ra} \left|\lim_{n \geq 1}(\Phi^{(n)})^*(\sU^{\bu})\right|_{\rm Stk} \overset{\varphi'}{\ra} \lim_{n \geq 1}(\Phi^{(n)})^*(\sX)
\]
where $\sU^{\bu}$ denotes the \v{C}ech nerve of $\sU \ra \sX$. Notice that $\lim_{n \geq 1}(\Phi^{(n)})^*(\sU^{\bu})$ is simply the \v{C}ech nerve of
\begin{equation}
    \label{eq:etale-surjection-at-level-n}
    \lim_{n \geq 1}(\Phi^{(n)})^*(\sU) \overset{v}{\ra} \lim_{n \geq 1}(\Phi^{(n)})^*(\sX),
\end{equation}
since we are only commuting limits. And that $v$ is an \'etale surjection, since given any point $\Spec\,k' \ra (\Phi^{(n)})^*(\sX)$ for some $n$, there exists an \'etale morphism $t^{(n),'} \ra t^{(n)}$ and a point $t^{(n),'} \ra \sX$, thus up to passing to an \'etale cover $t^{(n)''} \ra t^{(n)'}$ one has a section
\[
    \begin{tikzcd}
    t^{(n)''} \ar[r] \ar[d] & \sU \ar[d] \\
    t^{(n)'} \ar[r] \ar[ru,dashed] & \sX
    \end{tikzcd}
\]
Taking the pullback $t^{(n)''}\underset{t^{(n)'}}{\times}\Spec\,k' \ra t$ one has an \'etale cover of $\Spec\,k'$ where a section of $(\Phi^{(n)})^*(\sU) \ra (\Phi^{(n)})^*(\sX)$ exists. 

So we are reduced to checking that $\varphi$ is an equivalence. This would follow from each of the morphisms
\[
    (\Phi^{\infty})^{*}(\sU^m) \ra \lim_{n \geq 1}(\Phi^{(n)})^{*}(\sU^m)
\]
being an equivalence. Since each $\sU^m$ is $(n-1)$-geometric, by induction we are reduced to the case $n=0$. Now we notice that the morphism
\[
    \sU^{m}(T_0) \simeq \sU\underset{\sX}{\times}\cdots\underset{\sX}{\times}\sU  \ra \sU^{\times m}(T_0)
\]
is a monomorphism for every classical affine scheme $T_0$, so we can reduced to the case where one has a monomorphism $\classical{\sX} \ra \classical{\sZ}$ where $\sZ$ is the disjoint union of affine schemes, i.e.\ $\sX$ is $(-1)$-geometric. By again considering the \v{C}ech nerve in this situation we are reduced to the case where each $\sU^m$ is a disjoin union of affine schemes.

Now the claim is clear (cf.\ last paragraph the proof of Proposition 17.3.4.2 \cite{SAG}).
\end{proof}


Though condition (iv) is sufficient to obtain $n$-geometric stacks, we will prove that $n$-geometric actually satisfy the stronger (iv)' that is part of the deformation theory package. We start by discussion convergence.

\begin{prop}
\label{prop:n-gometric-is-convergent}
Let $\sX$ be an $n$-geometric stack, then $\sX$ is convergent.
\end{prop}

\begin{proof}
Again we want to apply induction. It is clear that any affine scheme is convergent. 

However, to set up the inductive set up for $n$-geometric stacks one needs a trick. The main point is that in trying to check convergence directly for $\sX$ one tries to commute an infinite limit with respect to a geometric realization of its \v{C}ech nerve. To circumvent that one notices that to check that $\sX$ is $n$-geometric it is enough to check that each of its truncations ${^{\leq k}\sX}$ for $k\geq 0$ is an $n$-geometric stack\footnote{Essentially, one just repeats the definition of \S \ref{subsubsec:0-geometric-defn} and \S \ref{subsubsec:n-geometric-defn} for the category ${^{\leq n}\PStk} :=\mbox{Fun}(({^{\leq k}\Schaff})^{\rm op},\Spc)$.} in prestacks defined over $k$-coconnective affine schemes, i.e.\ functors $({^{\leq k}\Schaff})^{\rm op} \ra \Spc$. Thus, though our initial set up for $n$-geometric stacks is slightly different than $n$-Artin stacks, exactly the same argument as in \cite[Proposition 4.4.9]{GRI} setting up a double induction works.
\end{proof}


In the following result we follow the same argument as in Lemma 7.4.3 from \cite[Chapter 1]{GRII}, which proves that any $n$-Artin stack admits deformation theory.

\begin{prop}
\label{prop:n-geometric-has-deformation-theory} 
Let $\sX$ be an $n$-geometric stack, then $\sX$ admits corepresentable deformation theory.
\end{prop}

\begin{proof}
We already know that $\sX$ is convergent by Proposition \ref{prop:n-gometric-is-convergent}, by Proposition \ref{prop:affine-homogeneous-is-deformation-theory} it is enough to check that $\sX$ is $W^{\rm all}$-homogeneous. 

Since $\sX$ is $n$-geometric there exists a smooth atlas $f:\sU \ra \sX$ such that $\sU$ and $f$ are $(n-1)$-geometric, so admit deformation theory by the inductive hypothesis. By Proposition \ref{prop:affine-homogeneous-is-deformation-theory} we need to check that for any pushout diagram as (\ref{eq:P-nil-square}) where $f:T_1 \ra T_2$ is arbitrary and $T_1 \hra T'_1$ is a nilpotent embedding that one has an isomorphism
\[
    \sX(T'_2) \ra \sX(T'_1)\underset{\sX(T_1)}{\times}\sX(T_2).
\]

In other words, the canonical morphism $\Maps_{T_2/-}(T'_2,\sX) \ra \Maps_{T_1/-}(T'_1,\sX)$ is an equivalence. Consider $\sU^{\bu}$ the \v{C}ech nerve of $f$ one obtains a diagram
\begin{equation}
\label{eq:inductive-deformation-theory-diagram}
    \begin{tikzcd}
    \left|\Maps_{T_2/-}(T'_2,\sU^{\bu})\right| \ar[r] \ar[d] & \Maps_{T_2/-}(T'_2,\sX) \ar[d] \\
    \left|\Maps_{T_1/-}(T'_1,\sU^{\bu})\right| \ar[r] & \Maps_{T_1/-}(T'_1,\sX)
    \end{tikzcd}
\end{equation}
We claim that the horizontal morphisms are monomorphisms of spaces. Indeed, suppose that we have two objects $u_1,u_2 \in \sU(T'_1)$ such that $f(u_1) \simeq f(u_2)$, then one has a morphism
\[
    T'_1 \ra \sU\underset{\sX}{\times}\sU \simeq \sX\underset{\sX\times \sX}{\times}\sU\times\sU,
\]
which gives a morphism between $u_1 \ra u_2$ and so produces the same object in the geometric realization $\left|\Maps(T'_1,\sU^{\bu})\right|$. 

Now, we notice that the left vertical morphism in (\ref{eq:inductive-deformation-theory-diagram}) is an equivalence. Indeed, $\sU\underset{\sX}{\times}\sU \ra \sU\times\sU$ is $(n-1)$-geometric, as the pullback of the diagonal of $\sX$, and the claim follows by the inductive hypothesis. 

Thus, we will be done if we check that the horizontal morphisms of (\ref{eq:inductive-deformation-theory-diagram}) are also surjections on $\pi_0$. However, any diagram for $m\geq 1$
\[
    \begin{tikzcd}
    T_1 \ar[d] \ar[r] & \sU^{m} \ar[d,"f^m"] \\
    T'_1 \ar[r] \ar[ur,dashed] & \sX
    \end{tikzcd}
\]
admits a lift since $T_1 \hra T'_1$ is a nilpotent embedding and $f^m$ is formally smooth. So we are reduced to the case of $n=0$. This case can be dealt with an argument similar to the proof of Proposition \ref{prop:n-geometric-is-integrable} above. Alternatively, this is Corollary 17.2.5.4 in \cite{SAG}.
\end{proof}

\begin{lem}
\label{lem:n-geometric-n-truncated}
Let $\sX$ be an $n$-geometric stack, then $\classical{\sX}$ is $n$-truncated. 
\end{lem}

\begin{proof}
We notice that $\classical{\sX}$ is $n$-truncated if
\[
    \classical{(\Delta_{\sX,S})}: \classical{\sX} \ra \classical{(\sX\underset{S}{\sX})}
\]
is $(n-1)$-truncated. Thus, by induction we are reduced to the case of $n=0$, but by our conventions \S \ref{subsubsec:0-geometric-defn} any $0$-geometric stack is $0$-truncated.
\end{proof}

\begin{rem}
\label{rem:n-truncated-and--n-connective}
We notice that if $\sX$ is a prestack that admits a cotangent complex such that $\classical{\sX}$ is $n$-truncated, then for any affine scheme point $x:T \ra \sX$ one has
\[
    T^*_{x}\sX \in \QCoh(T)^{\leq n},
\]
i.e.\ $T^*\sX$ is $(-n)$-connective. Indeed, by \cite[Chapter 1, \S 3.1.6]{GRII} we know that $T^*\sX$ is $(-n)$-connective if and only if for any $\sF \in \QCoh(T)^{\heartsuit}$ the space $\Maps_{T/}(T_{\sF},\sX)$ is $n$-truncated. Notice that if $\sF \in \QCoh(T)^{\heartsuit}$ one has a pushout diagram
\[
    \begin{tikzcd}
    \classical{T} \ar[r] \ar[d] & \classical{T_{\sF}} \ar[d] \\
    T \ar[r] & T_{\sF},
    \end{tikzcd}
\]
which $\sX$ takes to the following pullback diagram of spaces
\[
    \begin{tikzcd}
    \Maps(T_{\sF},\sX) \ar[r] \ar[d] & \Maps(\classical{T_{\sF}},\sX) \ar[d] \\
    \Maps(T,\sX) \ar[r] & \Maps(\classical{T},\sX)
    \end{tikzcd}
\]
Thus, one has 
\begin{equation}
    \label{eq:classical-and-derived-fibers}
    \Maps_{T/}(T_{\sF},\sX) \simeq \Maps_{\classical{T}/}(\classical{T_{\sF}},\sX)
\end{equation}
where the right-hand side of (\ref{eq:classical-and-derived-fibers}) is $n$-truncated as the fiber of two $n$-truncated spaces.
\end{rem}

\begin{warning}
Notice that the converse to Remark \ref{rem:n-truncated-and--n-connective} is not true, for instance if we consider the $1$-geometric stack $\sX:= BG$, for any reductive algebraic group $G$, then $\sX_{\rm dR}$ has vanishing cotangent complex, so in particular, is $n$-connective for any $n \in \bZ$, but $\sX_{\rm dR}(R) = BG(R^{\rm red.})$ is not $0$-truncated.
\end{warning}

\begin{proof}[Proof of Main Theorem: (i-vii) are necessary]
Assume that $\sX$ is an $n$-geometric locally almost of finite presentation stack. 

Conditions (i) and (ii) are part of the definition.

Condition (iii) is Proposition \ref{prop:n-geometric-is-integrable}.

Proposition \ref{prop:n-geometric-has-deformation-theory} implies conditions (iv)' and (v-vi), which imply (iv-vi).

Condition (vii)$_{n}$ is Lemma \ref{lem:n-geometric-n-truncated}.
\end{proof}

\subsection{\texorpdfstring{\'E}{E}tale surjection}

We start with an observation that will be allow us to proceed by induction on the proof of the main result. For that we introduce some terminology to streamline the discussion.

\begin{defn}
Given a morphism of prestacks $f:\sX \ra \sY$ and an integer $k \geq -2$, we say that $\classical{f}:\classical{\sX}\ra \classical{\sY}$ is $k$-truncated if for every classical affine scheme $T_0$ one has
\begin{equation}
    \label{eq:fiber-k-truncated}
    \Fib(\sX(T_0) \ra \sY(T_0))
\end{equation}
is a $k$-truncated space\footnote{Recall that, by definition, $(-2)$-truncated means contractible and $(-1)$-contractible means either empty or contractible.}.
\end{defn}

\begin{rem}
\label{rem:n-truncated-structure-map}
Notice that for a prestack $p:\sX \ra S$ and an integer $n\geq 0$, $\sX$ satisfies condition $(vii)_n$ if and only if $\classical{p}$ is $n$-truncated. Indeed, this follows from considering the long exact sequence of homotopy groups associated to (\ref{eq:fiber-k-truncated}) and noticing $\classical{S}$ is $0$-truncated for any affine scheme $S$ (see \cite[Chapter 2, \S 1.8.6]{GRI}).
\end{rem}

\begin{defn}
\label{defn:n-good}
Given a morphism $f:\sX \ra \sY$ of prestacks, we say that\footnote{We apologize for the terribly non-descriptive name that we chose for this property and warn the reader that the use of $n$-good here is different than that in \cite[\S 7.3]{DAG} in which condition (iv)' is imposed instead of (iv).} \emph{$n$-good} if for every affine scheme $S \ra \sY$ the base change morphism
\[
    \sX\underset{\sY}{\times}S \ra S
\]
satisfies conditions (i-vi) from Theorem \ref{thm:main-result}, and $\classical{(\sX\underset{\sY}{\times}S)} \ra \classical{S}$ is $n$-truncated.
\end{defn}


Notice that Definition \ref{defn:n-good} makes sense for $n \geq -2$. The proof of Theorem \ref{thm:main-result} will be by induction on $n$, but we need to make precise what the statement means for $n=-2$ and $n=-1$.

\begin{convention}
We will say that a prestack $\sX$ is 
\begin{itemize}
    \item \emph{$(-2)$-geometric} if $\sX$ is representable by an affine scheme;
    \item \emph{$(-1)$-geometric} if there exists an open embedding of $\sX$ into an affine scheme.
\end{itemize}
\end{convention}

Before starting the proof we introduce some notation to keep track of the iterations of passing to the diagonal morphism.

\begin{notation}
Given a morphism $f:\sX \ra \sY$ of prestacks and $S \ra \sX \underset{\sY}{\times}\sX$ a morphism from an affine scheme $S$ we will often consider the pullback square
\[
    \begin{tikzcd}
    D_{\sX/\sY,S} \ar[r,"(\Delta_{\sX/\sY})_{S}"] \ar[d] & S \ar[d] \\
    \sX \ar[r,"\Delta_{\sX/\sY}"] & \sX \underset{\sY}{\times}\sX
    \end{tikzcd}
\]
\end{notation}

The following observation is useful in performing the inductive steps in the proofs below.

\begin{prop}
\label{prop:n-good-morphism-has-n-1-good-diagonal}
Consider an $n$-good prestack $p:\sX \ra S$, then for any affine scheme $U\ra \sX\underset{S}{\times}\sX$ the diagonal morphism
    \[
        (\Delta_{\sX/S})_U: D_{\sX/S,U} \ra U
    \]
is $(n-1)$-good.
\end{prop}

\begin{proof}
First we notice that all the conditions (i-vii) and (iv') are stable under finite limits. Indeed, (i), (iii), (iv), (v), (vi) all follow formally since we are just commuting limits with each other.

For (ii) we notice that given $n\geq 0$ and a co-filtered diagram $(T_i)_I$ of $n$-coconnective affine schemes with limit $T\simeq \lim_I T_i$, one has
\begin{equation}
    \label{eq:n-truncation-is-not-necessary}
    \tau^{\leq n}(D_{\sX/S,U})(T) \overset{\simeq}{\ra}\tau^{\leq n}\left(\tlen{U}\underset{\tlen{\sX\underset{S}{\times}\sX}}{\times}\tlen{\sX}\right)(T),
\end{equation}
since a co-filtered limit of $n$-coconnective affine schemes is automatically $n$-coconnective, one obtains that the right-hand side of (\ref{eq:n-truncation-is-not-necessary}) is equivalent to 
\[
    \tlen{U}\underset{\tlen{\sX\underset{S}{\times}\sX}}{\times}\tlen{\sX}(T) \simeq
    \tlen{U(T)}\underset{\tlen{\sX\underset{S}{\times}\sX}(T)}{\times}\tlen{\sX}(T)
\]
Since $U,\sX$ and $\sX\underset{S}{\times}\sX$ are all locally almost of finite presentation we can rewrite the above as
\[
    \colim_{I^{\rm op}}\tlen{U}(T_i)\underset{\tlen{\sX\underset{S}{\times}\sX}(T_i)}{\times}\tlen{\sX}(T_i)
\]
where we used that filtered colimits commute with finite limits. Putting these two isomorphisms together we obtain that $\tlen{D_{\sX/S,U}}$ sends co-filtered limits to colimits.

Finally, for (vii) we notice that for any classical affine scheme $T_0$ one has a pullback of spaces
\[
    \begin{tikzcd}
    D_{\sX/S,U}(T_0) \ar[r] \ar[d] & U(T_0) \ar[d] \\
    \sX(T_0) \ar[r] & (\sX\underset{S}{\times}\sX)(T_0).
    \end{tikzcd}
\]

Then $D_{\sX/S,U}(T_0)$ is $(n-1)$-truncated, since the fiber of the lower horizontal morphism is $(n-1)$-truncated. Indeed, by considering the decomposition of the identity $\sX(T_0) \overset{\beta}{\ra} \sX\underset{S}{\times}\sX(T_0) \overset{\alpha}{\ra} \sX(T_0$ in the category of spaces, one has the following long-exact of homotopy groups
\[
    \cdots \ra \pi_{k+1}(\Fib(\alpha)) \ra \pi_k(\Fib(\beta)) \ra \pi_k(\Fib(\alpha \circ \beta)) \ra \pi_k(\Fib(\alpha)) \ra \pi_{k-1}(\Fib(\beta)) \ra \cdots,
\]
where the middle terms vanish, since $\Fib(\alpha\circ \beta)$ is contractible.
\end{proof}

The \hyperref[thm:main-result]{Main Theorem} is proved by conveniently exploiting the following result due to J.\ Lurie. It provides us with enough points almost of finite presentation which are formally smooth over an $n$-good prestack, i.e.\ one constructs an \'etale surjection from a disjoint union of affine schemes into our prestack.

\begin{thm}
\label{thm:etale-surjection}
Let $p:\sX \ra S$ be a prestack over an excellent affine scheme, such that $p$ satisfies conditions:
\begin{enumerate}
    \item $\sX$ is a sheaf with respect to the \'etale topology;
    \item $\sX$ is locally almost of finite presentation;
    \item $\sX$ is integrable;
    \item[$(iv)'_{\rm all}$] $\sX$ is infinitesimally cohesive;
    \item[$(v)^{\rm laft}_{ft}$] $\sX$ admits a cotangent complex at every point $(T \overset{x}{\ra} \sX) \in (\Schaffconvft)_{/\sX}$ such that
    \[
        T^*_x\sX \in \Pro(\QCoh(T)^-)_{\rm laft};
    \]
    \item[(vi)] $\sX$ is convergent;
    \item[$(vii)_n$] $\classical{\sX}$ is $n$-truncated. 
\end{enumerate}
Let $I$ denote the subcategory of $\Schaff_{/\sX}$ generated by formally smooth morphisms $U \ra \sX$ such that $U \ra S$ is almost of finite presentation. Consider the cover 
\[
f:\sU := \bigsqcup_{I}U_i \ra \sX.
\]
Then $f$ is an \'etale surjection.
\end{thm}

\begin{proof}

\textit{Step 1:} $f$ is an \'etale surjection if the following holds: for every affine scheme point $x:T \ra \sX$ there exists an \'etale cover $T' \ra T$ such that one can find a factorization
\begin{equation}
    \label{eq:etale-cover-lift}
    \begin{tikzcd}
    T' \ar[r,dashed] \ar[d] & U_i \ar[d] \\ 
    T \ar[r] & \sX
    \end{tikzcd}
\end{equation}
where $(f_i: U_i \ra \sX) \in I$.

\textit{Step 2 (reduction to classical points):} since $f_i$ is formally smooth, in any diagram
\[
    \begin{tikzcd}
    \classical{T} \ar[r,dashed] \ar[d] & U_i \ar[d] \\
    T \ar[r] & \sX
    \end{tikzcd}
\]
we can fill the dashed arrow. Thus, it is enough to consider the case of classical affine scheme points $T \in \clSchaff_{/\sX}$.

\textit{Step 3 (reduction to point almost of finite type):} since $\sX$ is locally almost of finite presentation, given a diagram $T \simeq \lim_{I}T_i$ presenting the classical affine scheme $T$ as a (co-)filtered limit of affine schemes of finite presentation over $S$, the point $x$ is determined by a point $x_i:T_i \ra \sX$. Thus, we will assume that $T$ is of finite presentation over $S$.

\textit{Step 4 (reduction to \'etale neighborhoods):} suppose that for each point $t \in |T|$ there exists an \'etale neighborhood of $t$, i.e.\ $V_t \ra T$ \'etale whose image contains $t$ such that one has a lift
\[
    \begin{tikzcd}
    V_t \ar[r,dashed] \ar[d] & \sU \ar[d] \\
    T \ar[r] & \sX
    \end{tikzcd}
\]
then one can find lifts for any diagram as in (\ref{eq:etale-cover-lift}). Indeed, one can simply take the \'etale cover $\bigsqcup_{t \in |T|}V_t \ra T$.

\textit{Step 5 (reduction to Henselinisations):} Let $T^{\rm H.}_t$ be the Henselisation of $T$ at the point $t$. Concretely, we can write $T^{\rm H.}_T$ as a limit over some \'etale neighborhoods of $t$\footnote{Which ones is not relevant for the argument, but those whose induced morphism between the residue fields is an isomorphism (cf.\ \cite[Chapitre 8, Proposition 2 and Th\'eor\`eme 1]{MR0277519}).}, i.e.\ there exists a filtered set $I$ and for each $i \in I$ an \'etale neighborhood $V_i \ra T$ of $t$ such that 
\[
    T^{\rm H.}_t \simeq \lim_{I^{\rm op}} V_i.
\]
Assume that a lift as follows exists
\[
    \begin{tikzcd}
    T^{\rm H.}_t \ar[r,"x^{\rm H.}_t"] \ar[d] & \sU \ar[d] \\
    T \ar[r] & \sX
    \end{tikzcd}
\]
Then, since $\sX$ is locally almost of finite presentation, $\classical{\sX}$ takes limits to colimits, so there exists $i \in I$ and a morphism $x_{i}: V_i \ra \sX$ such that one has a lift
\[
    \begin{tikzcd}
    T^{\rm H.}_t \ar[r,"x^{\rm H.}_t"] \ar[d] & \sU \ar[d,"\id_{\sU}"] \\
    V_i \ar[r,"x_i"] \ar[d] & \sU \ar[d] \\
    T \ar[r] & \sX
    \end{tikzcd}
\]

\textit{Step 6 (lift of closed point):} we will construct a lift for the closed point of $T^{\rm H.}_t$, let $x_t: \Spec(\kappa(t)) \ra T^{\rm H.}_t \ra \sX$ denote the restriction of $x$ to the residue field of $t \in |T|$. By Step 4, we know that $T \ra S$ is finitely presented, hence the composite $p\circ x_{t}:\Spec(\kappa(t)) \ra T \ra S$ factors as follows
\[
    \begin{tikzcd}
    \Spec(\kappa(t)) \ar[r] \ar[d] & T \ar[d] \\
    \Spec(\kappa(s)) \ar[r] & S
    \end{tikzcd} 
\]
where $s \in |S|$ and $\kappa(s) \ra \kappa(t)$ is a finitely generated field extension\footnote{Indeed, given any factorization 
\[
    \begin{tikzcd}[ampersand replacement=\&]
    k[S] \ar[r] \ar[d] \& k[T] \ar[d] \\
    k \ar[r] \& \kappa(t)
    \end{tikzcd}
\]
and a set of elements $S$ such that $\kappa(t) \subset k(S)$, since $k[T]$ is Henselian we can lift the generators of $\kappa(t)$ to generators of $k[T]$ over $k[S]$, which could then be reduced to a finite number, by the finite presentation assumption.
}. Thus, we can apply Proposition \ref{prop:approximate-chart} to obtain an ``approximate chart'', that is, a factorization of $x_t$ as follows
\[
    \begin{tikzcd}
    \Spec(\kappa(t)) \overset{x_{t,B}}{\ra} \Spec(B) \overset{g}{\ra} \sX
    \end{tikzcd}
\]
where $\Spec(B) \ra S$ is almost of finite presentation and $H^{-1}(T^*_{x'_t}(g)) = 0$.

\textit{Step 7:} Now we by Proposition \ref{prop:smooth-chart} there exists a non-empty Zariski neighborhood $\Spec\, B' \ra \Spec\, B$ such that the composite $g':\Spec\, B' \ra \Spec\, B \ra \sX$ is formally smooth and $\Spec\, B' \ra S$ is almost of finite presentation. In other words, there exists a commutative diagram as follows
\[
    \begin{tikzcd}
    & \Spec\, B' \ar[r] \ar[d] & \sU \ar[d] \\
    \Spec\, \kappa(t) \ar[r,"x_{t,B}"] & \Spec\, B \ar[r,"g"] & \sX,
    \end{tikzcd}
\]
which gives a morphism 
\[
    \Spec\, \kappa'(t) := \Spec\, \kappa(t)\underset{\Spec\, B}{\times}\Spec\,B' \ra \Spec\, B'.
\]
Since $\Spec\, \kappa'(t) \ra \Spec\, \kappa(t)$ is an non-empty Zariski neighborhood, up to localizing $T$ further from the beginning we can assume that $\kappa(t) \overset{\simeq}{\ra}\kappa'(t)$. Thus, we have a lift
\[
    \begin{tikzcd}
    \Spec\, \kappa(t) \ar[r,"x_{t,B}"] \ar[d] & \sU \ar[d] \\
    T^{\rm H.}_t \ar[r] & \sX
    \end{tikzcd}
\]

\textit{Step 8 (extending the lift from a closed point):}

\textit{Step 8 (i) (lift to a formal completion):}
Now we only need to lift $x_{t,B}$ to $T^{\rm H.}_t$. For that consider $T^{\wedge}_t$ the completion of $T^{\rm H.}_t$ at the closed point $t$. One has a canonical factorization $\Spec(\kappa(t)) \ra T^{\wedge}_t \ra T^{\rm H.}_t$. Since $\sX$ is integrable any morphism $T^{\wedge}_t \ra \sX$ is determined by a map $T^{(n)}_t \ra \sX$ for some $n \geq 1$, and since $\Spec(\kappa(t)) \ra T^{(n)}_t$ is a nilpotent embedding one obtains a lift $T^{(n)}_t \ra \sU$ because $\sU \ra \sX$ is formally smooth. Thus, we have a diagram
\[
    \begin{tikzcd}
    T^{\wedge}_t \ar[r,"\hat{x}_{t}"] \ar[d] & \sU \ar[dd] \\
    T^{(n)}_t \ar[ru,"x^{(n)}_t"] \ar[d] & \\
    T^{\rm H.}_t \ar[r] & \sX
    \end{tikzcd}
\]
where the morphism $x^{(n)}_t$ is the lift determined by the formal smoothness condition.

\textit{Step 8 (ii) (lift to an actual smooth cover):} to lift the morphism $\hat{x}_t$ to $T^{\rm H.}_t$ we use Proposition \ref{prop:derived-Popescu}. Notice that $T^{\wedge}_t \ra T^{\rm H.}_t$ is flat morphism between excellent rings\footnote{Indeed, $T^{\rm H.}_t$ is excellent since $T \ra S$ is of finite presentation and $S$ is excellent hence $T$ is excellent, and Henselisation preserves excellent rings (see \cite[Expos\'e 1, \S 8]{illusieTravauxGabberUniformisation2012}). $T^{\wedge}_t$ is excellent since it is the completion of $T^{\rm H.}_t$ with respect to an ideal that is contained in its Jacobson ideal (see \cite[Expos\'e 1, Proposition 9.1]{illusieTravauxGabberUniformisation2012}).} then it is geometrically regular, so the derived Popescu's theorem \cite[Theorem 3.7.5]{DAG} says that there exists a filtered diagram $J$ and a presentation
\[
    T^{\wedge}_T \overset{\simeq}{\ra} \lim_{J^{\rm op}}Z_j
\]
where $\pi_j: T_j \ra T^{\rm H.}_t$ is smooth for each $j \in J$. Because $\sX$ is almost of finite presentation there exists a lift
\[
    \begin{tikzcd}
    T_j \ar[r,"x_j"] \ar[d,"\pi_j"'] & \sU \ar[d] \\
    T^{\rm H.}_t \ar[r] & \sX
    \end{tikzcd}
\]

\textit{Step 8 (iii) (lift to the Henselisation):} Finally, since $\pi_j$ is smooth and it has a section over the closed point $t \in T^{\rm H.}_t$ it has a section over $T^{\rm H.}_{t}$ so we get a lift
\[
    \begin{tikzcd}
    T_j \ar[r,"x_j"] \ar[d,"\pi_j"'] & \sU \ar[d] \\
    T^{\rm H.}_t \ar[r] \ar[u,"s_{j}",dashed,bend left=60] & \sX
    \end{tikzcd}
\]

This finishes the proof!

\end{proof}

\subsection{Proof of main theorem}

Our proof will use induction, exploiting the following result of J.\ Lurie.

\begin{thm}
\label{thm:classical-representable}
Let $\sX$ be a prestack
\begin{condlist}
    \item assume that $\sX$ satisfy the following subset of conditions:
        \begin{enumerate}
        \item[$(iv)'_{\rm all}$] $\sX$ is infinitesimally cohesive;
        \item[$(v)_{\rm ec}$] $\sX$ admits a cotangent complex at every $(S\overset{x}{\ra}\sX) \in \Schaffconv_{/\sX}$;
        \item[(vi)] $\sX$ is convergent.
        \end{enumerate}
    \item Moreover, assume that there exists a $0$-geometric stack $\sY$ and an equivalence
    \[
        \classical{\sX} \overset{\simeq}{\ra} \classical{\sY}.
    \]
\end{condlist}
Then $\sX$ is $0$-geometric.
\end{thm}

\begin{proof}
This is \cite[Theorem 18.1.0.2]{SAG}. We notice that in our terminology $0$-geometric stands for spectral Deligne--Mumford stack whose underlying classical stack is $0$-truncated. Theorem 18.1.0.2 states that $\sX$ is just a spectral Deligne--Mumford stack, however since $\classical{\sX} \simeq \classical{\sY}$ the truncatedness claim follows from the definition of $0$-geometric.
\end{proof}

The following observation is useful in the inductive step of the argument below:

\begin{rem}
\label{rem:n-geometric-at-the-level-of-prestacks}
Given a morphism $f:\sX \ra \sY$ of prestacks, if $f$ is $n$-geometric, then the morphism between the associated \'etale stacks\footnote{Here $L:\PStk \ra \Stk$ is the left adjoint to the inclusion of \'etale stacks in prestacks.}
\[
    L(F):L(\sX) \ra L(\sY)
\]
is also $n$-geometric.
\end{rem}

\begin{proof}[Proof of main theorem: (i-vii) are sufficient]

\

\textit{Step 1 (extending induction):} By Remark \ref{rem:n-truncated-structure-map} we notice that condition $(vii)_n$ in Theorem \ref{thm:main-result} is equivalent to requiring that the morphism $\classical{p}:\sX \ra S$ is $n$-truncated. This last condition makes sense for $n\geq -2$ and we perform induction with the base case being $n = -2$.

\textit{Step 2 (smooth atlas):} Let $f:\sU \ra \sX$ denote the smooth cover of $\sX$ constructed in Proposition \ref{thm:etale-surjection} and denote by $\sU^{\bu}$ the simplicial object obtained by taking the \v{C}ech nerve of $f$. For any $m\geq 0$ we have
\[
    \sU^m \simeq \bigsqcup_{i \in I^{m+1}}U_{i_1}\underset{\sX}{\times}\cdots\underset{\sX}{\times}U_{i_{m+1}}
\]
where each $U_{i_j}$ is an affine scheme. We notice that one has the following pullback diagram
\[
    \begin{tikzcd}
    U_{i_1}\underset{\sX}{\times}\cdots\underset{\sX}{\times}U_{i_{m+1}} \ar[r,"p"] \ar[d] & U_{i_1}\underset{S}{\times}\cdots\underset{S}{\times}U_{i_{m+1}} \ar[d] \\
    \sX \ar[r,"\Delta_{\sX/S}"'] & \sX\underset{S}{\times}\sX
    \end{tikzcd}
\]
So by Proposition \ref{prop:n-good-morphism-has-n-1-good-diagonal} one has that
\[
    U_{i_1}\underset{\sX}{\times}\cdots\underset{\sX}{\times}U_{i_{m+1}} \ra S
\]
is $(n-1)$-good. Thus, by the inductive hypothesis each $U_{i_1}\underset{\sX}{\times}\cdots\underset{\sX}{\times}U_{i_{m+1}}$ is $(n-1)$-geometric.

\textit{Step 3:} Notice that each $\sU^m$ is $(n-1)$-geometric as a disjoint union of $(n-1)$-geometric stacks. Since by construction $\sU \ra \sX$ is an \'etale surjection, one obtains
\[
    \left|\sU^{\bu}\right|_{\Stk} \overset{\simeq}{\ra} \sX.
\]
 
Notice that we will be done if we check that $\sX \ra \sX\underset{S}{\times}\sX$ is $(n-1)$-geometric. By Remark \ref{rem:n-geometric-at-the-level-of-prestacks} it is enough to check that 
\[
    \left|\sU^{\bu}\right|_{\PStk} \ra \left|\sU^{\bu}\right|_{\PStk} \underset{S}{\times} \left|\sU^{\bu}\right|_{\PStk}
\]
is $(n-1)$-geometric. Given $T \ra \left|\sU^{\bu}\right|_{\PStk} \underset{S}{\times} \left|\sU^{\bu}\right|_{\PStk}$ by definition this factor through $\sU \underset{S}{\times}\sU$ so considering the diagram 
\begin{equation}
    \label{eq:pullback-via-diagonal-is-level-2-Cech}
    \begin{tikzcd}
    \left|\sU^{\bu}\right|_{\PStk} \underset{\left|\sU^{\bu}\right|_{\PStk} \underset{S}{\times} \left|\sU^{\bu}\right|_{\PStk}}{\times}T \ar[r] \ar[d] & T \ar[d] \\
    \left|\sU^{\bu}\right|_{\PStk}\underset{\left|\sU^{\bu}\right|_{\PStk} \underset{S}{\times} \left|\sU^{\bu}\right|_{\PStk}}{\times}\sU \ar[r]\ar[d] & \sU\underset{S}{\times}\sU \ar[d] \\
    \left|\sU^{\bu}\right|_{\PStk} \ar[r] & \left|\sU^{\bu}\right|_{\PStk} \underset{S}{\times} \left|\sU^{\bu}\right|_{\PStk}
    \end{tikzcd}
\end{equation}
We notice that $\left|\sU^{\bu}\right|_{\PStk}\underset{\left|\sU^{\bu}\right|_{\PStk} \underset{S}{\times} \left|\sU^{\bu}\right|_{\PStk}}{\times}\sU \simeq \sU\underset{|\sU^{\bu}|_{\PStk}}{\times}\sU \simeq \sU\underset{\sX}{\times}\sU$, so the upper horizontal map in (\ref{eq:pullback-via-diagonal-is-level-2-Cech}) is $(n-1)$-geometric, since $\sU\underset{\sX}{\times}\sU \ra \sU \underset{S}{\times}\sU$ is $(n-1)$-geometric.

\textit{Step 4 (base step):} We notice that in the case $p:\sX \ra S$ is $(-2)$-truncated, i.e.\ $\classical{\sX} \overset{\simeq}{\ra}\classical{S}$, the diagonal morphism $\sX \overset{\simeq}{\ra} \sX \underset{S}{\times}\sX$ is an isomorphism, so in particular it is $0$-geometric and Lemma \ref{lem:main-lemma} implies that $\sX$ is infinitesimally cohesive, i.e.\ satisfies conditions (iv)', (v), and (vi). Thus, Theorem \ref{thm:classical-representable} implies that $\sX$ is $0$-geometric and we are done.

\end{proof}

\printbibliography

\end{document}